\documentclass[12pt]{amsart}

\usepackage{amsmath,amssymb,amsthm,upref}
\usepackage[margin=1.5in]{geometry}                
\usepackage[ansinew]{inputenc}
\usepackage{mathrsfs}
\usepackage[all]{xy}
\usepackage{fullpage}
\usepackage{paralist}
\usepackage{graphicx}

\newcommand{\Dl}{\ensuremath{\Delta} }
\newcommand{\simp}{\Delta^\comp}

\newcommand{\ds}{\displaystyle}

\newcommand{\mc}[1]{\mathcal{#1}}
\newcommand{\mrm}[1]{\mathrm{#1}}
\newcommand{\mbf}[1]{\mathbf{#1}}

\newcommand{\mtt}[1]{\mathtt{#1}}
\newcommand{\mbb}[1]{\mathbb{#1}}
\newcommand{\comp}{{\scriptscriptstyle \ensuremath{\circ}}}
\newcommand{\compc}{{\scriptscriptstyle \ensuremath{\stackrel{\star}{\mbox{}}}}}

\newcommand{\scr}[1]{{\scriptscriptstyle #1}}

\theoremstyle{theorem}
\newtheorem{thm}{Theorem}[section]

\newtheorem{lema}[thm]{Lemma}
\newtheorem{prop}[thm]{Proposition}
\newtheorem{cor}[thm]{Corollary}

\theoremstyle{definition}
\newtheorem{ej}[thm]{Example}

\newtheorem{obs}[thm]{Remark}
\newtheorem{defi}[thm]{Def{i}nition}

\newenvironment{num}{\medskip
\refstepcounter{thm}\noindent {\bf (\thethm)}}{\vspace{0.2ex}\par}

\begin{document}

\title[Realizable homotopy colimits]
{Realizable homotopy colimits}
%
%
\markboth{Beatriz Rodr\'iguez Gonz\'alez}{Realizable homotopy colimits}
%
%
 \author[ Beatriz Rodr\'iguez Gonz\'alez]{ Beatriz Rodr\'iguez-Gonz\'alez}
 \address{\hspace{-0.5cm}  B. Rodr\'iguez-Gonz\'alez
    \newline Instituto de Ciencias Matem\'aticas
    \newline
    Campus Cantoblanco UAM 
    \newline
    28049 Madrid
    \newline
    SPAIN }
 \email{rgbea@icmat.es}

%
%
%
%
%
 \thanks{Partially supported by the ERC Starting Grant project TGASS and the contracts MTM2007-67908-C02-02, FQM-218 and SGR-119. 
The author gratefully acknowledges STORM department of London Metropolitan University 
for excellent working conditions.}

%

\begin{abstract}
In this paper we prove that for any model category, the Bousf{i}eld-Kan construction of the homotopy colimit is 
the absolute left derived functor of the colimit. This is achieved by showing that the Bousf{i}eld-Kan homotopy colimit 
is  moreover a \textit{realizable homotopy colimit}, def{i}ned 
by means of a suitable 2-category of \textit{relative categories}. In addition, in the case of exact coproducts, we characterize the realizable homotopy 
colimits that satisfy a cof{i}nality property as those given by a formula following the pattern of Bousf{i}eld-Kan construction: they are the composition 
of a `geometric realization' with the simplicial replacement.
\end{abstract}

\maketitle

\renewcommand{\contentsname}{Table of Contents.}
\tableofcontents 


\section*{Introduction.}

Homotopy colimits are a cornerstone of Algebraic Topology and Homotopy Theory, but they are presented under dif{f}erent approaches in the 
literature. Sometimes a homotopy colimit is def{i}ned through a specif{i}c construction: for instance, the Bousf{i}eld-Kan homotopy colimit 
is a explicit formula f{i}rst introduced in \cite{BK} for simplicial sets, and generalized to general model categories (\cite{H}).
Other times a homotopy colimit is introduced through a universal property: in the Quillen model category setting, homotopy colimits are def{i}ned as 
left derived functors of colimits. Also, in the setting of Grothendieck derivators a homotopy colimit is def{i}ned as a left adjoint to the 
localized constant diagram functor. More recently, in the context of $\Dl$-closed classes V. Voevodsky def{i}nes a homotopy colimit as a particular 
construction involving the simplicial replacement of diagrams (\cite{V}). And there are other treatments of homotopy colimits, for instance the one of \cite{DHKS}.

In this paper we prove that the specif{i}c constructions of Bousf{i}eld-Kan and Voevodsky 
satisfy the universal properties of Grothendieck and Quillen homotopy colimits. 
With this aim, we work in the general framework of relative categories 
of \cite{BaK} (or, equivalently, of `homotopical categories' under the terminology of \cite{DHKS}).
A natural notion of homotopy colimit in this framework is the one of 
\textit{realizable homotopy colimit}, that we define, broadly speaking, as a homotopy colimit `realized' by a functor 
defined \textit{before} localizing with respect to the weak equivalences. This is made precise 
using a suitable 2-category structure on relative categories (see Definition \ref{defiRhocolim}). 
It holds that, after localizing with 
respect to the weak equivalences, a realizable homotopy colimit is a Grothendieck homotopy colimit and 
an absolute left derived functor of the colimit.\\
\indent In the case of exact coproducts, we characterize the realizable homotopy colimits that satisfy a cof{i}nality property as those given by a formula 
following the pattern of Bousf{i}eld-Kan and Voevodsky constructions.
More precisely, they are the composition of a `geometric realization' for simplicial objects with the simplicial replacement. For a relative category 
$(\mc{C},\mc{W})$ the property of possessing a `geometric realization' $\mbf{s}:\simp\mc{C}\rightarrow\mc{C}$, that we call \textit{simple functor}, is 
encoded in the notion of \textit{simplicial descent category} (\cite{R}), which is inspired by the cubical homological descent categories 
of \cite{GN}. Under this notations, the characterization previously mentioned is the\\[0.2cm]
\textbf{Theorem \ref{characterization}.}
\textit{Let $(\mc{C},\mc{W})$ be a relative category closed by coproducts. The following are equivalent:
\begin{compactitem}
\item[\textbf{i.}] $(\mc{C},\mc{W})$ admits realizable homotopy colimits $\mtt{hocolim}_I:\mc{C}^I\rightarrow \mc{C}$, which
are invariant under homotopy right cof{i}nal changes of diagrams.
\item[\textbf{ii.}] $(\mc{C},\mc{W})$ admits a simplicial descent structure with simple functor $\mbf{s}:\simp\mc{C}\rightarrow \mc{C}$.
\end{compactitem}
In addition, if these equivalent conditions hold then for each small category $I$ there is a natural 
isomorphism $\mtt{hocolim}_I \simeq \mbf{s}\amalg^I$ of $\mc{R}el\mc{C}at$, 
where $\amalg^I:\mc{C}^I\rightarrow \simp\mc{C}$ denotes the simplicial replacement.}\\[0.2cm]
Furthermore, we prove that under the previous equivalent conditions all homotopy left Kan extensions exist on $(\mc{C},\mc{W})$ 
and may be computed pointwise. In the setting of Grothendieck derivators, this means that the prederivator associated with $(\mc{C},\mc{W})$ 
is a \textit{weak right derivator} (see Theorem \ref{RightDerivator}).
Then, in the case of exact coproducts we have reduced the task of checking the existence of homotopy colimits and pointwise homotopy left Kan extension on a 
given $(\mc{C},\mc{W})$ to the task of verifying that $(\mc{C},\mc{W})$ admits a simple functor $\mbf{s}:\simp\mc{C}\rightarrow \mc{C}$.

In the context of model categories, we show that the Bousf{i}eld-Kan simplicial homotopy colimit is a simple functor on pointwise cof{i}brant diagrams.
We use this fact together with previous characterization to prove the\\[0.2cm] 
\textbf{Theorem \ref{MCHCbis}.}
\textit{Let $(\mathcal{M},\mathcal{W})$ be a model category. Then $(\mc{M},\mc{W})$ admits realizable homotopy colimits, which are invariant under homotopy right cof{i}nal changes of diagrams. 
In addition, they may be computed using the corrected 
Bousf{i}eld-Kan formula
$$\mtt{{}_c hocolim}^{BK}_I X = \int^i  \widetilde{QX}(i) \otimes \mrm{N}(i/I)^{\comp}$$}
\noindent As a corollary, we deduce that the corrected Bousf{i}eld-Kan homotopy colimit on any model category is indeed the absolute left derived functor of the colimit.
A proof of this fact for the particular case of a simplicial model category may be found in \cite{Ga}.
Another consequence of previous theorem is that a functor between model categories that preserves weak equivalences between cof{i}brant objects 
commutes with homotopy colimits if and only if it commutes with homotopy 
coproducts and with the homotopy colimit for simplicial objects (see Corollary \ref{FcommutesHocolim}).

Returning to the general setting of realizable homotopy colimits, we remark that Voevodsky homotopy colimits for $\Dl$-closed classes are not only a particular 
instance of them, but a key ingredient in the proof of previous results. Indeed, it is proved in \cite{R} that a simple functor gives an equivalence
in $\mc{R}el\mc{C}at$ between a simplicial descent category $(\mc{C},\mc{W})$ and $(\simp\mc{C},\mc{S})$, where $\mc{S}$ is a $\Dl$-closed class.
Under this equivalence the realizable homotopy colimit $\mbf{s}\amalg^I$ on $(\mc{C},\mc{W})$ corresponds to the Voevodsky homotopy colimit on 
$(\simp\mc{C},\mc{S})$. Using simplicial homotopy theory, and remarkably Illusie's bisimplicial decalage (\cite{I}), 
we construct in section \ref{VoevodskySection} explicit adjunction morphisms to prove that Voevodsky homotopy colimit is realizable.\\

\noindent {\sc Acknowledgements:} I would like to show my gratitude to Luis Narv{\'a}ez Macarro and Vicente Navarro Aznar for encouraging me to continue in the 
research world. This work would not be possible without their mathematical and professional advice. I am also deeply 
indebted to Javier Fern{\'a}ndez de Bobadilla for his conf{i}dence and patience during my participation in his ERC Starting Grant research project.

\section{Simplicial Preliminaries.}

\subsection{Simplicial objects.}\mbox{}\\[0.2cm]
\indent We denote by $\Dl$ the \textit{simplicial category}, with objects the ordered sets $[n]=\{0<\ldots < n\}$, $n\geq 0$, and morphisms the order
preserving maps. The \textit{face maps} $d^i:[n-1]\rightarrow [n]$ are characterized by $d^i([n-1])=[n]-\{i\}$, and the degeneracy maps $s^j:[n+1]\rightarrow [n]$ are the
surjective monotone maps with $s^j(j)=s^j(j+1)$. They satisfy the well-known \textit{simplicial identities}, and generate all maps in $\Dl$ (see \cite{May}).
We will also write $[n]$ for the category associated with the ordered set $\{0< 1 < \cdots < n\}$.

By $\simp\mc{D}$ (resp. $\simp\simp\mc{D}$) we mean the category of \textit{simplicial} (resp. \textit{bisimplicial}) \textit{objects} in a f{i}xed category $\mc{D}$. 
The \textit{constant simplicial object} $c(A)$ def{i}ned by an object $A$ of $\mc{D}$ is is the simplicial object equal to $A$ in each degree and with identities as 
face and degeneracy maps. In this way we obtain the constant functor $c:\mc{D}\rightarrow \simp\mc{D}$. When understood, we will
denote also $c(A)$ by $A$. The \textit{diagonal functor} $\mrm{D}:\simp\simp\mc{D}\rightarrow\simp\mc{D}$ is given by 
$\mrm{D}(\{Z_{n,m}\}_{n,m\geq 0})=\{Z_{n,n}\}_{n\geq 0}$.

Dually, $\Dl\mc{D}$ is the category of \textit{cosimplicial objects} in $\mc{D}$. Since $\Dl\mc{D}=(\simp\mc{D}^\comp)^\comp$, all
def{i}nitions and constructions concerning simplicial objects may be dualized to the cosimplicial setting.
\vspace{0.2cm}
\subsection{Simplicial homotopies.}\mbox{}\\[0.2cm]
\indent If $\mc{D}$ has coproducts, there is a natural action
$\simp\mc{D}\times \simp Set \rightarrow \simp\mc{D}$, $(X,K)\mapsto X\otimes K$, given by
\begin{equation}\label{ActionSset}(X\otimes K)_n=\coprod_{K_n}X_n\end{equation}
We also denote by $\otimes$ the induced action $\mc{D}\times \simp Set \rightarrow \simp\mc{D}$, $A\otimes K= c(A) \otimes K$.
Note that only f{i}nite coproducts are needed to construct $X\otimes K$ in case $K$ is a \textit{simplicial f{i}nite set}, that is, in case each 
$K_n$ is a f{i}nite set.\\ 
Recall that $\Dl[k]$ is the simplicial f{i}nite set with $\Dl[k]_n=Hom_{\Dl}([n],[k])$. Given a simplicial object $X$, then $X\otimes \Dl[1]$ is the \textit{simplicial cylinder} of $X$. The maps
$d^0,d^1:[0]\rightarrow [1]$ induce $d_0^X,d_1^X:X\rightarrow X\otimes \Dl[1]$. Simplicial homotopies and simplicial homotopy equivalences are def{i}ned in $\simp\mc{D}$ as usual. That is,
$f,g:X\rightarrow Y$ are called \textit{simplicially homotopic}, denoted $f\sim g$, if there exists $H:X\otimes\Dl[1]\rightarrow Y$ such that $H d_0^X =f$ and $H d_1^X =g$. A map
$f:X\rightarrow Y$ is said to be a \textit{simplicial homotopy equivalence} if there exists $f':Y\rightarrow X$ such that $f f' \sim 1$ and $f' f\sim 1$.

Given an object $A$ of $\mc{D}$ and a simplicial object $X$, an augmentation $\epsilon:X\rightarrow A$ is just a simplicial morphism $\epsilon:X\rightarrow c(A)$. It holds that $\epsilon$ is a simplicial homotopy equivalence if and only if $\epsilon$ has an
\textit{extra degeneracy} ${s}_{-1}$ or $s_{n+1}$ (see \cite[3; 3.2]{B}). This means that there exists $s_{-1}:A\rightarrow X_0$ (resp. ${s}_{0}:A\rightarrow X_0$) and 
$s_{-1}:X_{n}\rightarrow X_{n+1}$ (resp. $s_{n+1}:X_{n}\rightarrow X_{n+1}$) satisfying the simplicial identities.\\
Dually, if $\mc{D}$ has products there is a natural coaction $\Dl\mc{D}\times (\simp Set)^{\comp}\rightarrow \Dl\mc{D}$, $(X,K)\mapsto X^K$,
$$(X^K)^n =\prod_{K_n}X^n$$
In this way, the cosimplicial path object of $X$ is def{i}ned as $X^{\Dl[1]}$, and cosimplicial homotopies and homotopy equivalences are def{i}ned dually.
\vspace{0.2cm}
\subsection{Undercategories and overcategories.}\mbox{}\\[0.2cm]
\indent Given a functor $F:I\rightarrow J$ and an object $x\in J$, recall that the \textit{overcategory} $(F/ x)$ has as objects the pairs $(y,f)$
where $y\in {I}$ and $f:F(y)\rightarrow x$ of $J$. The morphisms of $(F/ x)$ are the commutative triangles
$$\xymatrix@M=4pt@H=4pt@C=25pt{ F(y) \ar[r]^{f} \ar[d]_{F(g)} & x  \\
F(y') \ar[ru]_{f'} & }$$
We denote by $(I/ x)$ the overcategory $(1_{I}/ x)$, where $1_I:I\rightarrow I$ is the identity functor. Analogously, the morphisms $x\rightarrow F(y)$ form the \textit{undercategory} $(x/F)$.
\vspace{0.2cm}
\subsection{Simplicial nerve and cof{i}nality.}\mbox{}\\[0.2cm]
\indent Denote by $cat$ the category of small categories. Given $I$ in $cat$, its \textit{simplicial nerve} $\mrm{N}(I)$ is the simplicial set given in degree $n$ by 
$Hom_{cat}([n],I)$. In other words, an $n$-simplex of $\mrm{N}(I)$ is a functor from $[n]$ to $I$. The face and degeneracy maps of $\mrm{N}(I)$ are obtained by 
composition with $d^i : [n-1]\rightarrow [n]$ and $s^j:[n]\rightarrow [n+1]$.

A functor $F:I\rightarrow J$ is called \textit{left cof{i}nal} if for each $x\in J$ the simplicial set $\mrm{N}(F/x)$ is non-empty and connected.
It is \textit{homotopy left cof{i}nal} if $\mrm{N}(F/ x)$ is  contractible for each $x\in J$. This means that $\mrm{N}(F/x)\rightarrow \Dl[0]$ is a weak homotopy 
equivalence of simplicial sets.\\
Dually $F$ is \textit{right cof{i}nal}, resp. \textit{homotopy right cof{i}nal}, if $F^{op}:{I}^{op}\rightarrow J^{op}$ is left cof{i}nal, resp. homotopy left cof{i}nal.
This is the same as saying that for each $x\in J$, $\mrm{N}(x/F)$ is non-empty and connected (resp. contractible).
\vspace{0.2cm}
\subsection{Simplicial replacement of diagrams.}\mbox{}\\[0.2cm]
\indent Let $\mc{D}$ be a category closed by coproducts and $I$ a small category, and denote by $\mc{D}^I$ the category of functors from $I$ to $\mc{D}$.
The \textit{simplicial replacement functor} $ \amalg {}^{I} :\mc{D}^I \longrightarrow  \simp\mc{D} $ maps $X:I\rightarrow \mc{D}$ to the simplicial object 
$\amalg^I X$ given in degree $n$ by
 $$\amalg^I_n X= \ds\coprod_{i_0\rightarrow \cdots\rightarrow i_n} X_{i_0} $$
The coproduct is indexed over the set consisting of those $\underline{i}=\{ i_0\rightarrow \cdots\rightarrow i_n\}\in\mrm{N}_n(I)$, the $n$-simplexes of the 
simplicial nerve of $I$. The face and degeneracy maps of $\amalg^I X$ are def{i}ned as follows. If $0< k \leq n$ then $d_k : \amalg^I_n X\rightarrow \amalg^I_{n-1} X$ maps
the term $X_{i_0}$ with index $\underline{i}$ to the term $X_{i_0}$ with index $d_k(\underline{i})$ through the identity, while $d_0$ sends this term to
$X_{i_1}$ with index $d_0(\underline{i})$, through $X(i_0\rightarrow i_1)$. The degeneracy map $s_k : \amalg^I_n X\rightarrow \amalg^I_{n+1} X$
sends $X_{i_0}$ with index $\underline{i}$ to $X_{i_0}$ with index $s_k(\underline{i})$ through the identity.

Note that the simplicial replacement construction is natural on $I$. Given two diagrams $X:I\rightarrow \mc{D}$ and $Y:J\rightarrow \mc{D}$, a morphism
$(f,\tau):X\rightarrow Y$ between them is a functor $f:I\rightarrow J$ plus a natural transformation $\alpha :  X\rightarrow f^\ast Y$.
Then 
\begin{equation}\label{SRnatural}\amalg^\bullet (f,\tau): \amalg^I X \rightarrow \amalg^J Y\end{equation} 
sends the component $X_{i_0}$ indexed by $\underline{i}$ to the component $Y_{f(i_0)}$ indexed by $f(\underline{i})$ through 
$\tau_{i_0}:X_{i_0}\rightarrow (f^\ast Y)_{i_0}=Y_{f(i_0)}$.

\begin{obs}
The colimit of a diagram $X:I\rightarrow \mc{D}$, if it exists, agrees with the colimit of $\amalg^I X$. Indeed,  the colimit of the simplicial diagram $\amalg^I X$ is the coequalizer of $\coprod_{i} X_{i} \leftleftarrows \coprod_{i\rightarrow j} X_{i}$, which computes the colimit of $X$.
Therefore, $\amalg^I X$ has in this case the augmentation
\begin{equation}\label{Augmcolim}\xymatrix@M=4pt@H=4pt@C=25pt{ \mtt{colim}_I X  &  \coprod_{i} X_{i} \; \ar[l] \ar@/_1.5pc/[r] & {\;} \coprod_{i\rightarrow j} X_{i} \;
\ar@<0.5ex>[l] \ar@<-0.5ex>[l] \ar@/_1.5pc/[r]
\ar@/_1pc/[r] & \; \coprod_{i\rightarrow j\rightarrow k} X_{i}\; \ar@<0ex>[l] \ar@<1ex>[l]
\ar@<-1ex>[l] & \cdots\cdots }\end{equation}
\end{obs} 
\vspace{0.2cm}
\subsection{Illusie's decalage bisimplicial construction.}\mbox{}\\[0.2cm]
\indent We will strongly use Illusie's decalage bisimplicial construction introduced in \cite[p.7]{I}.
Recall that the \textit{bisimplicial decalage} $dec:\simp\mc{D}\rightarrow\simp\simp\mc{D}$ is induced by the ordinal sum $\Dl\times\Dl\rightarrow \Dl$, 
$[n]+[m]=[n+m+1]$. Given a simplicial object $Y$, $dec(Y)$ is the bisimplicial object given in bidegree $(n,m)$ by 
$dec(Y)_{n,m}=Y_{n+m+1}$.  
The face and degeneracy maps of $dec(Y)$ are def{i}ned as follows. On one hand $d_k^{I} : dec(Y)_{n,m}\rightarrow dec(Y)_{n-1,m}$  is $d_k : Y_{n+m+1}\rightarrow Y_{n+m}$, while $s_k^I$ is $s_k:Y_{n+m+1}\rightarrow
Y_{n+m+2}$. On the other hand $d_k^{II} : dec(Y)_{n,m}\rightarrow dec(Y)_{n,m-1}$ is $d_{n+k+1}:Y_{n+m+1}\rightarrow Y_{n+m}$, and $s^{II}$ is $s_{n+k+1}:Y_{n+m+1}\rightarrow
Y_{n+m+2}$.\\
Denote by $Y\times\Dl$ and $\Dl\times Y$ the bisimplicial objects with
$(Y\times\Dl)_{n,m}=Y_n$ and $(\Dl\times Y)_{n,m}=Y_m$. There are two natural augmentations $\Lambda^{I}: dec(Y)\rightarrow \Dl\times Y$ and $\Lambda^{II}:dec(Y)\rightarrow Y\times\Dl$ given respectively by
$\Lambda^{I}_{0,m}=d_0 : Y_{m+1}\rightarrow Y_m$ and $\Lambda^{II}_{n,0}=d_{n+1} : Y_{n+1}\rightarrow Y_n$.

We will make use of the following result.
\begin{prop}\label{IlDec}\emph{(\cite[Proposition 1.6.2]{I})} For each simplicial object $Y$ of $\simp\mc{D}$, the diagonals of
the augmentations $\Lambda^I$ and $\Lambda^{II}$, $\mrm{D}(\Lambda^I),\mrm{D}(\Lambda^{II}):\mrm{D}(dec(Y))\rightarrow Y$ are
simplicial homotopy equivalences which are in addition simplicially homotopic.
\end{prop}

In addition, it follows from the proof given in loc. cit. that the simplicial homotopies involved in previous proposition are natural on $Y$.

\section{Realizable homotopy colimits.}

\subsection{Relative categories as a 2-category.}\mbox{}\\[0.2cm]
\indent
Given a class $\mc{W}$ of morphisms in a category $\mc{C}$, recall that the \textit{localization} of $\mc{C}$ with respect to $\mc{W}$ is the result of formally 
inverting the morphisms of $\mc{W}$ in $\mc{C}$. This gives a (possibly big category) $\mc{C}[\mc{W}^{-1}]$ plus a localization functor
$\gamma:\mc{C}\rightarrow \mc{C}[\mc{W}^{-1}]$ sending the elements of $\mc{W}$ to isomorphisms, and inducing for each category $\mc{E}$ an equivalence
of categories
\begin{equation}\label{localization}- \compc \gamma: Fun(\mc{C}[\mc{W}^{-1}],\mc{E}) \longrightarrow Fun_{\mc{W}}(\mc{C},\mc{E})\end{equation}
Here $Fun(\mc{C}[\mc{W}^{-1}],\mc{E})$ is the category of functors $F':\mc{C}[\mc{W}^{-1}]\rightarrow\mc{E}$ and 
$Fun_{\mc{W}}(\mc{C},\mc{E})$ is the full subcategory of $Fun(\mc{C},\mc{E})$ of functors $F:\mc{C}\rightarrow\mc{E}$ that send the elements in 
$\mc{W}$ to isomorphisms.

\begin{defi}
A \textit{relative category} consists of a pair $(\mc{C},\mc{W}_{\mc{C}})$ formed by a category $\mc{C}$ and a class of morphisms
$\mc{W}_{\mc{C}}$ of $\mc{C}$, whose elements are called \textit{weak equivalences}. The class $\mc{W}_{\mc{C}}$, also denoted by $\mc{W}$ for brevity, is 
assumed to be saturated. That is, $\mc{W}$ is the inverse image  by the localization functor $\gamma:\mc{C}\rightarrow\mc{C}[\mc{W}^{-1}]$ of the isomorphisms of 
$\mc{C}[\mc{W}^{-1}]$. 
\end{defi}

\begin{defi}
We say that a relative
category $(\mc{C},\mc{W})$ is \textit{closed by $($f{i}nite$)$ coproducts} if $\mc{C}$ has an initial object $0$ and both $\mc{C}$ and $\mc{W}$ are closed by
(f{i}nite) coproducts. Note that in this case $\mc{C}[\mc{W}^{-1}]$ is again closed by (f{i}nite) coproducts, and they 
are preserved by $\gamma:\mc{C}\rightarrow \mc{C}[\mc{W}^{-1}]$.
\end{defi}

If $(\mc{C},\mc{W})$ is a relative category and $I$ is any category, the category $Fun(I,\mc{C})$ of functors from $I$ to $\mc{C}$, also denoted
by $\mc{C}^I$ for shortness, is again a relative category with the \textit{pointwise weak equivalences}. These are by def{i}nition the natural transformations $\tau:F\rightarrow G$  
such that $\tau_i:F(i)\rightarrow G(i)$ is a weak equivalence of $\mc{C}$ for each object $i$ of $I$. The pointwise weak equivalences are easily seen 
to form again a saturated class. They will be denoted by $\mc{W}^{I}$, or just by $\mc{W}$ if there is no risk of confusion. 

\begin{defi}
We consider the 2-category $\mc{R}el\mc{C}at$ of relative categories, whose 1-morphisms and 2-morphisms are respectively the relative functors and relative 
natural transformations def{i}ned as follows.
A \textit{relative functor} $F:(\mc{C},\mc{W})\rightarrow (\mc{D},\mc{W})$ is a weak equivalence preserving functor $F:\mc{C}\rightarrow\mc{D}$. 
That is, $F$ maps a weak equivalence of $\mc{C}$ to a weak equivalence of $\mc{D}$.\\ 
A  \textit{relative natural transformation} between the relative functors $F,G:(\mc{C},\mc{W})\rightarrow (\mc{D},\mc{W})$ is a morphism  
$\tau:F\dashrightarrow G$ in $Fun(\mc{C},\mc{D})[\mc{W}^{-1}]$, the category of relative functors from $\mc{C}$ to $\mc{D}$ localized by 
the pointwise weak equivalences. More precisely, $\tau$ is represented by a f{i}nite zigzag connecting $F$ and $G$
$$ F \cdots \bullet \rightarrow \bullet\leftarrow \bullet \rightarrow \bullet \cdots G $$
formed by functors and natural transformation between them, such that those natural transformations going to the left are 
pointwise weak equivalences.
\end{defi}

\begin{lema} With the above notions of relative functors and relative natural transformations, $\mc{R}el\mc{C}at$ is a \emph{2}-category. 
\end{lema}

\begin{proof}  The composition of relative functors is just the usual composition of functors. To def{i}ne a 2-category structure on $\mc{R}el\mc{C}at$ it remains 
to def{i}ne the compositions $\tau\compc F$ and $G\compc \tau$ of a relative natural transformation $\tau : T\dashrightarrow L$ between the relative functors 
$T,L:(\mc{C},\mc{W})\rightarrow (\mc{D},\mc{W})$, and the  relative functors $F:(\mc{C}',\mc{W})\rightarrow (\mc{C},\mc{W})$, $G:(\mc{D},\mc{W})\rightarrow 
(\mc{D}',\mc{W})$. The usual compositions between natural transformations and functors of categories are functors
$$\begin{array}{cccc} -\compc F :&  Fun(\mc{C},\mc{D}) &  \longrightarrow & Fun(\mc{C}',\mc{D})\\
			      & T  & \mapsto &  T\comp F\\
				& \alpha: T\rightarrow L & \mapsto & \alpha\compc F  \\   
    G\compc - :& Fun(\mc{C},\mc{D}) & \longrightarrow & Fun(\mc{C},\mc{D}')\\
& T  & \mapsto &  G \comp T \\
				& \alpha: T\rightarrow L & \mapsto & G\compc \alpha 
  \end{array}
$$
Note that $-\compc F$ always preserves pointwise weak equivalences, while $G\compc -$ preserves them because $G$ is a relative functor. Therefore 
they induce
$$\begin{array}{c} -\compc F : Fun(\mc{C},\mc{D})[\mc{W}^{-1}] \longrightarrow Fun(\mc{C}',\mc{D})[\mc{W}^{-1}]\\
    G\compc - : Fun(\mc{C},\mc{D})[\mc{W}^{-1}] \longrightarrow Fun(\mc{C},\mc{D}')[\mc{W}^{-1}]
  \end{array}
$$    
which def{i}ne compositions of relative natural transformations and relative functors. As in the case of usual categories, these compositions satisfy by def{i}nition 
the requested compatibility axioms, so $\mc{R}el\mc{C}at$ is a 2-category.
\end{proof}

\begin{lema}\label{loc2func}  
Localization by weak equivalences  
$$\mrm{loc}:\mc{R}el\mc{C}at\rightarrow \mc{C}at\ \ , \ \ (\mc{C},\mc{W}) \mapsto \mc{C}[\mc{W}^{-1}]$$ 
is a \emph{2}-functor from relative categories to categories.
\end{lema}

\begin{proof}
To see that $loc$ is a 2-functor, note that a relative functor $F:(\mc{C},\mc{W})\rightarrow (\mc{D},\mc{W})$ 
induces by the universal property of localization a unique functor $loc(F): \mc{C}[\mc{W}^{-1}]\rightarrow  \mc{D}[\mc{W}^{-1}]$ such that 
$loc(F)\, \gamma = \gamma\, F$. This def{i}nes $loc$ on 1-morphisms. To def{i}ne it on 2-morphisms, consider relative functors $F,G:(\mc{C},\mc{W})\rightarrow (\mc{D},\mc{W})$
and $\tau:F\dashrightarrow G$ in $Fun(\mc{C},\mc{D})[\mc{W}^{-1}]$. If $c$ is an object of $\mc{C}$, then $\tau_c : Fc\dashrightarrow Gc$ may be thought of as
a morphism in $\mc{D}[\mc{W}^{-1}]$, which is natural on $c$ by construction. More precisely, the functor $\gamma\compc - : Fun(\mc{C},\mc{D})\rightarrow 
Fun(\mc{C},\mc{D}[\mc{W}^{-1}])$ sends a pointwise weak equivalence to an isomorphism, inducing $\gamma\compc - : Fun(\mc{C},\mc{D})[\mc{W}^{-1}]\rightarrow 
Fun(\mc{C},\mc{D}[\mc{W}^{-1}])$. Therefore, $\gamma\compc \tau: \gamma\comp F\rightarrow\gamma\comp G$ is a morphism of 
$Fun(\mc{C},\mc{D}[\mc{W}^{-1}])$. Since $F$ and $G$ are relative functors, $\gamma\comp F$ and $\gamma\comp G$ send
a weak equivalence of $\mc{C}$ to an isomorphism of $\mc{D}[\mc{W}^{-1}]$. This means that $\gamma\compc\tau$ is a morphism of 
$Fun_{\mc{W}}(\mc{C},\mc{D}[\mc{W}^{-1}])$, and it correspond through the equivalence of categories (\ref{localization}) to the desired 
$loc(\tau): loc(F)\rightarrow loc(G)$.
\end{proof}

\begin{lema}\label{exp2func}
 Given a small category $I$, exponentiation by $I$ 
$$(-)^I:\mc{R}el\mc{C}at\rightarrow \mc{R}el\mc{C}at \ \ , \ \ (\mc{C},\mc{W}) \mapsto (\mc{C}^I,\mc{W}^I)$$ 
is a \emph{2}-functor. Here, $\mc{W}^I$ denotes the class of pointwise weak equivalences of $\mc{C}^I$.
\end{lema}

\begin{proof}
If $F:(\mc{C},\mc{W})\rightarrow (\mc{D},\mc{W})$ is a relative functor then
$F^I:(\mc{C}^I,\mc{W}^I)\rightarrow (\mc{D}^I,\mc{W}^I)$,  given by $(F^I(X))(i)=F(X(i))$,
is clearly a relative functor. On the other hand, if $\tau:F\dashrightarrow G$ is a relative natural transformation,
$\tau^I:F^I\dashrightarrow G^I$ is also def{i}ned pointwise. This is possible because 
$(-)^I:Fun(\mc{C},\mc{D})\rightarrow Fun(\mc{C}^I,\mc{D}^I)$ preserves pointwise weak equivalences, inducing 
$$(-)^I:Fun(\mc{C},\mc{D})[\mc{W}^{-1}]\rightarrow Fun(\mc{C}^I,\mc{D}^I)[(\mc{W}^I)^{-1}]$$
\end{proof}

A relative functor $F:(\mc{C},\mc{W})\rightarrow (\mc{D},\mc{W})$ is a \textit{relative equivalence} if it has a quasi-inverse in $\mc{R}el\mc{C}at$.
More explicitly, there exists a relative functor $G:(\mc{D},\mc{W})\rightarrow (\mc{C},\mc{W})$ 
and relative natural transformations $FG\dashrightarrow 1_{\mc{D}}$, $GF\dashrightarrow 1_{\mc{C}}$ which are invertible in $Fun(\mc{D},\mc{D})[\mc{W}^{-1}]$ 
and $Fun(\mc{C},\mc{C})[\mc{W}^{-1}]$ respectively. In this case we say that $(\mc{C},\mc{W})$ and $(\mc{D},\mc{W})$ are \textit{equivalent relative categories}.
\vspace{0.2cm}
\subsection{Relative adjunctions.} 

\begin{defi}
A \textit{relative adjunction} $(F,G,\alpha,\beta)$ between the relative categories 
$(\mc{C},\mc{W})$ and $(\mc{D},\mc{W})$ is an adjunction in the 2-category $\mc{R}el\mc{C}at$. 
More concretely, it consists of:\\
$1$. Relative functors $F:(\mc{C},\mc{W})\rightarrow (\mc{D},\mc{W})$ and $G:(\mc{D},\mc{W})\rightarrow (\mc{C},\mc{W})$.\\
$2$. Relative natural transformations $\alpha:FG\dashrightarrow 1_{\mc{D}}$ and $\beta:1_{\mc{C}}\dashrightarrow GF$ satisfying the so called 
\textit{triangle identities}. That is, the compositions
$$\xymatrix@M=4pt@H=4pt@R=5pt@C=30pt{F\ar@{-->}[r]^{F\beta} & FGF \ar@{-->}[r]^{\alpha_F} & F \\
G\ar@{-->}[r]^{\beta_G} & GFG \ar@{-->}[r]^{F\alpha} & G}$$
are the identity in $Fun(\mc{C},\mc{D})[\mc{W}^{-1}]$
and $Fun(\mc{D},\mc{C})[\mc{W}^{-1}]$ respectively.\\
If $(F,G,\alpha,\beta)$, or $(F,G)$ for short, is a relative adjunction we say that $F$ is a 
left relative adjoint of (or relative left adjoint to) $G$ and that $G$ is a right relative adjoint of (or relative left adjoint to) $F$. 
\end{defi}

\begin{ej} If $(\mc{C},\mc{W})$ and $(\mc{D},\mc{W})$ are relative categories and $(F,G,\alpha,\beta)$ is an usual adjunction between
$\mc{C}$ and $\mc{D}$, then $(F,G,\alpha,\beta)$ is a relative adjunction provided that $F$ and $G$ preserve weak equivalences.
\end{ej}

\begin{obs}
In case the relative natural transformations $\alpha$ and $\beta$ are isomorphisms of $\mc{R}el\mc{C}at$ we say that
$(F,G)$ is a \textit{relative adjoint equivalence}. Although we will not use this fact, we remark that as happens in the context of usual categories,
a relative equivalence of relative categories gives rise to a relative adjoint equivalence. 
\end{obs}

Since 2-functors preserve adjunctions, the following facts follow from Lemmas \ref{loc2func} and \ref{exp2func}.

\begin{lema}\label{reladjLoc} The localization of a relative adjunction $F:(\mc{C},\mc{W})\leftrightarrows (\mc{D},\mc{W}):G$ is an adjunction 
$F:\mathcal{C}[\mc{W}^{-1}]\leftrightarrows \mathcal{D}[\mc{W}^{-1}]:G$ between the corresponding localized categories. 
\end{lema}

\begin{lema}\label{reladjDiag}  If $I$ is a small category and $F:(\mathcal{C},\mc{W})\leftrightarrows (\mathcal{D},\mc{W}):G$
is a relative adjunction then $F^I:(\mathcal{C}^I,\mc{W})\leftrightarrows (\mathcal{D}^I,\mc{W}):G^I$ is again a 
relative adjunction.
In particular, $F^I:\mathcal{C}^I[\mc{W}^{-1}]\leftrightarrows \mathcal{D}^I[\mc{W}^{-1}]):G^I$ is a natural adjunction for 
each small category $I$.
\end{lema}

The following two results are formal properties of
adjoints in 2-categories. 

\begin{prop} Let $F: (\mc{C},\mc{W})\rightarrow (\mc{D},\mc{W})$ and $G: (\mc{D},\mc{W})\rightarrow (\mc{C},\mc{W})$ be relative functors.
Then $(F,G)$ is a relative adjunction if and only if for each category $I$ and functors $T:I\rightarrow \mc{D}$, $T':I\rightarrow \mc{C}$ there is
a natural bijection
$$\mrm{Hom}_{Fun(I,\mc{D})[\mc{W}^{-1}]} ( FT', T )  \leftrightarrow \mrm{Hom}_{Fun(I,\mc{C})[\mc{W}^{-1}]}(T',GT)$$
\end{prop}

Under the above notations, if $\tau: FT' \dashrightarrow T$ is a relative natural transformation, the corresponding $\tau':T'\dashrightarrow GT$ 
will be called the \textit{adjoint relative natural transformation} of $\tau'$, and conversely $\tau'$ will be called the adjoint relative natural transformation 
of $\tau$.

\begin{prop}\label{UniqRelAd} Let $F,F':(\mc{C},\mc{W})\rightarrow (\mc{D},\mc{W})$ be relative
functors sharing the same relative right adjoint $G$. Then there exists a unique 
isomorphism $\tau:F\dashrightarrow F'$ of $Fun(\mc{C},\mc{D})[\mc{W}^{-1}]$ compatible
with the adjunction morphisms of $(F,G)$ and $(F',G)$.
\end{prop}
\vspace{0.1cm}
\subsection{Realizable homotopy colimits.}\mbox{}\\[0.2cm]
\indent
Given a small category $I$ and a relative category $(\mc{C},\mc{W})$, the \textit{constant diagram functor} $c_I:\mc{C}\rightarrow \mc{C}^I$ is def{i}ned by 
$(c_I(x))(i) = x$ for all $i\in I$. Note that it clearly sends a weak equivalence to a pointwise weak equivalence, so $c_I:(\mc{C},\mc{W})\rightarrow (\mc{C}^I,\mc{W})$
is a relative functor. 

\begin{defi}\label{defiRhocolim}
A \textit{realizable homotopy colimit} on a relative category $(\mc{C},\mc{W})$ is a relative adjunction $(\mtt{hocolim}_I,c_I,\alpha,\beta)$ between $(\mc{C}^I,\mc{W})$ and $(\mc{C},\mc{W})$.
By simplicity, we will often drop the adjunction morphisms and refer to a realizable homotopy colimit as the relative left adjoint
$\mtt{hocolim}_I : (\mc{C}^I,\mc{W})\rightarrow (\mc{C},\mc{W})$ of the constant diagram functor $c_I:(\mc{C},\mc{W})\rightarrow 
(\mc{C}^I,\mc{W})$. 
\end{defi}

\begin{obs} A \textit{realizable homotopy limit} is defined in the dual way. Throughout the paper we focus on realizable homotopy colimits, 
but all the constructions and results presented here can be dualized to the setting of realizable homotopy limits.
\end{obs}

\begin{obs}
Note that realizable homotopy colimits generalize usual colimits. If $\mc{C}$ is a category, recall that a colimit $\mtt{colim}:\mc{C}^I\rightarrow \mc{C}$
is by def{i}nition a left adjoint of $c_I:\mc{C}\rightarrow \mc{C}^I$, which is the same thing as a realizable homotopy colimit  
on $(\mc{C},\mc{W}=\{\mbox{isomorphisms}\})$.
\end{obs}

Since by Lemma \ref{reladjLoc} the localization of a relative adjunction gives a usual adjunction, 
the localization of a realizable homotopy colimit produces a left adjoint $\mtt{hocolim}_I : 
\mc{C}^{I}[\mc{W}^{-1}]\rightarrow \mc{C}[\mc{W}^{-1}]$ of the localized 
constant diagram functor. In other words, a realizable homotopy colimit gives, after localizing, a homotopy colimit in the sense of Grothendieck derivators. 
On the other hand, it is proved in \cite[Proposition 4.2]{RII} that a left adjoint of the localized 
constant diagram functor is an absolute left derived functor of $\mtt{colim}_I : 
\mc{C}^{I}\rightarrow \mc{C}$, in case it exists. Therefore we deduce the

\begin{prop}\label{DerivedFunctorColim}
 Let $(\mc{C},\mc{W})$ be a relative category. Assume that for a small category $I$ there exists the colimit $\mtt{colim}_I:\mc{C}^I\rightarrow\mc{C}$.
If $\mtt{hocolim}_I:(\mc{C}^I,\mc{W})\rightarrow (\mc{C},\mc{W})$ is a realizable homotopy colimit on $(\mc{C},\mc{W})$ then
the absolute left derived functor $\mbb{L}\mtt{colim}_I$ of $\mtt{colim}_I$ exists and it agrees with the functor 
$\mtt{hocolim}_I : \mc{C}^{I}[\mc{W}^{-1}]\rightarrow \mc{C}[\mc{W}^{-1}]$ induced by the realizable homotopy colimit on localizations.
\end{prop}

As in the case of colimits, the Fubini property of realizable homotopy colimits is a formal consequence of adjointness.

\begin{prop}\label{Fubini} Let $(\mc{C},\mc{W})$ be a relative category. Given small categories $I$ and $J$, there is a unique isomorphism of relative 
functors $$\mtt{hocolim}_{I\times J}  \dashrightarrow \mtt{hocolim}_I \mtt{hocolim}_J  $$
compatible with the adjunction morphisms.
\end{prop}

\begin{proof} We have that $\mtt{hocolim}_{I\times J}:(\mc{C}^{I\times J},\mc{W})\rightarrow (\mc{C},\mc{W})$ is relative 
left adjoint to $c_{I\times J}:(\mc{C},\mc{W})\rightarrow (\mc{C}^{I\times J},\mc{W})$. We claim that the same holds for 
$\mtt{hocolim}_{I}\mtt{hocolim}_{J}:(\mc{C}^{I\times J},\mc{W})\rightarrow (\mc{C},\mc{W})$.
Indeed, $\mtt{hocolim}_{I}:(\mc{C}^{I},\mc{W})\rightarrow (\mc{C},\mc{W})$ 
is relative left adjoint to $c_I$. On the other hand, relative adjunctions are inherited by taking diagram categories by Lemma \ref{reladjDiag}. Therefore 
$\mtt{hocolim}^{\mc{C}^I}_J:((\mc{C}^I)^J,\mc{W}^{-1})\rightarrow (\mc{C}^I,\mc{W}^{-1})$ is relative left adjoint to $c_J$. Since the composition of relative
adjunctions is a relative adjunction, it follows that
$\mtt{hocolim}_{I}\mtt{hocolim}_{J}$ is relative left adjoint to $c_Jc_I = c_{I\times J} : (\mc{C},\mc{W})\rightarrow (\mc{C}^{I\times J},\mc{W})$. We conclude
by Proposition \ref{UniqRelAd} that there is a unique isomorphism between $\mtt{hocolim}_{I}\mtt{hocolim}_{J}$ and $\mtt{hocolim}_{I\times J}$ compatible with the adjunction
morphisms.
\end{proof}

There are other properties of usual colimits which are also formal consequences of adjointness in the 2-category $\mc{C}at$, but it is not clear whether 
they hold or not for realizable homotopy colimits. That is, it is not clear if they are or not formal consequences of adjointness in the 2-category 
$\mc{R}el\mc{C}at$. 
For instance, colimits are known to be \textit{invariant under cof{i}nal changes of diagrams}. This means that given a right cof{i}nal functor between small categories
$f:I\rightarrow J$, the natural morphism $\mtt{colim}_I f^{\ast}\rightarrow \mtt{colim}_J$ induced on colimits is an isomorphism. In the relative category 
setting the corresponding property is the following.

\begin{num}\label{hrcdef{i}bis} 
 Given a functor $f:I\rightarrow J$ between small categories, we consider the relative natural transformation 
\begin{equation}\label{hocolimf}\mtt{hocolim}(f):\mtt{hocolim}_I f^{\ast}\dashrightarrow \mtt{hocolim}_J\end{equation}
induced on homotopy colimits by adjunction. More concretely, $\mtt{hocolim}(f)$ is the adjoint relative natural transformation of $f^{\ast} \beta : f^{\ast} 
\rightarrow f^{\ast} c_J \mtt{hocolim}_J = c_I \mtt{hocolim}_J$, where $\beta:1_{\mc{C}^J}\rightarrow c_J\mtt{hocolim}_J$ denotes the adjunction morphism
of the adjoint pair $(\mtt{hocolim}_J,c_J)$.
\end{num}

\begin{defi}\label{hrcdef{i}}
We say that realizable homotopy colimits on $(\mc{C},\mc{W})$ are \textit{invariant under homotopy right cof{i}nal changes of diagrams} if 
given a homotopy right cof{i}nal functor between small categories
$f:I\rightarrow J$, the induced relative natural transformation (\ref{hocolimf})
is an isomorphism of $\mc{R}el\mc{C}at$.
\end{defi}

In good situations, such as model categories, homotopy colimits are known to be invariant under homotopy right cof{i}nal changes of diagrams (see \cite[Theorem 19.6.7]{H}). 
Then, a natural question is whether the invariance under homotopy right cof{i}nal changes of diagrams is a formal property of realizable homotopy colimits or not.
For a Grothendieck derivator, it holds that homotopy colimits are also invariant under homotopy right cof{i}nal changes of diagrams (see \cite[Corollaire
1.14]{C}). Although the proof is formal, it uses not only the existence of homotopy colimits, but also the existence of certain homotopy limits.


\section{Voevodsky homotopy colimits are realizable.}\label{VoevodskySection}

We prove in this section that the Voevodsky formula for homotopy colimits introduced in \cite[p. 11]{V} produces realizable homotopy colimits.
These are def{i}ned for relative categories  $(\simp\mc{C},\mc{S})$  such that $\mc{S}$ is a $\Dl$-closed class closed by coproducts, 
and are constructed combining the diagonal with the classical simplicial replacement  $(\simp\mc{C})^I\rightarrow\simp(\simp\mc{C})$.
\vspace{0.2cm}
\subsection{Voevodsky homotopy colimits.}\mbox{}\\[0.2cm]
\indent
To begin with, we remind the def{i}nition and properties of $\Dl$-closed classes.

\begin{defi}(\cite{V})
Let $\mc{C}$ be a category closed by coproducts. A saturated class $\mc{S}$ of morphisms in $\simp\mc{C}$ is $\Dl$-\textit{closed} if it satisf{i}es the following properties:\\[0.1cm]
\textit{1.} The class $\mc{S}$ contains the simplicial homotopy equivalences.\\[0.1cm]
\textit{2.} If $F=F_{\cdot,\cdot}:Z_{\cdot,\cdot}\rightarrow T_{\cdot,\cdot}$ is a morphism of bisimplicial objects in $\mc{C}$ such that
$F_{n,\cdot}\in\mc{S}$ (or $F_{\cdot,n}\in\mc{S}$) for all $n\geq 0$, then the diagonal $\mrm{D}(F)$ of $F$ is in $\mc{S}$.\\
\indent A morphism $F:X\rightarrow Y$ in $\simp\mc{C}$ is a \textit{termwise coprojection} if for each $n\geq 0$ there exists an object $A^{(n)}$ of $\mc{C}$ and a 
commutative diagram
$$\xymatrix@M=4pt@H=4pt@C=25pt@R=10pt{ X_n \ar[r]^{F_n} \ar[rd] &  Y_n \ar@{-}[d]^{\wr} \\
				                         &  X_n \amalg A^{(n)} }$$
where $X_n\rightarrow X_n\amalg A^{(n)}$ is the canonical morphism.
\end{defi}

\begin{obs}\mbox{}\\
\textit{1}. We restrict ourselves to \textit{saturated}
$\Dl$-closed classes. In this case the previous def{i}nition agrees with the one given in \cite{V}.\\
\textit{2}. Note that hypothesis \textit{1} is equivalent to the fact that two simplicially homotopic morphisms of $\simp\mc{C}$ become equal in $\simp\mc{C}[\mc{S}^{-1}]$.
\end{obs}

\begin{lema}\label{DeltaCst} Let $(\simp\mc{C},\mc{S})$ be a relative category such that $\mc{S}$ is $\Dl$-closed. Let $f:X\rightarrow Y$ be a morphism of simplicial objects
such that for each $n\geq 0$ the constant simplicial morphism $f_n:X_n\rightarrow Y_n$ is in $\mc{S}$. Then $f$ is in $\mc{S}$ as well.
\end{lema}

\begin{proof} Consider the bisimplicial objects $X\times\Dl$ and $Y\times\Dl$ which are constant with respect to the second simplicial degree. That is, 
$(X\times \Dl)_{n,m} = X_n$ and analogously for $Y\times \Dl$. Then $f$ gives the bisimplicial morphism $f\times  \Dl$ def{i}ned in the same way. 
By assumption $(f\times  \Dl)_{n,\cdot}$ is in $\mc{S}$ for each $n\geq 0$, so its diagonal $\mrm{D}(f\times  \Dl)=f$ is in $\mc{S}$ as well.
\end{proof}

The good homotopical behavior of $\Dl$-closed classes is highlighted by the fact that it 
produces a natural structure of Brown category of cof{i}brant objects. This means, roughly speaking, that
pushouts along cof{i}brations exist and preserve weak equivalences, and that each object has a cylinder. 
Among the properties of Brown categories of cof{i}brant objects are the existence, up to homotopy, of a calculus of left fractions for the localized category, and of
cof{i}ber sequences enjoying the usual good properties.  The reader is referred to \cite{Br} for further detail. 

\begin{prop}\emph{(\cite[Proposition 4.9]{R})}\label{DlClBrown}
Let $(\simp\mc{C},\mc{S})$ be a relative category closed by f{i}nite coproducts, and such that $\mc{S}$ is $\Dl$-closed. Then
$(\simp\mc{C},\mc{S})$ is a Brown category of cof{i}brant objects, where the cof{i}brations are the termwise coprojections.
\end{prop}

\begin{defi}\label{def{i}HolimDl}\cite[p. 11]{V}
Let $\mc{C}$ be a category closed by coproducts. The \textit{Voevodsky homotopy colimit}
$$ \mtt{hocolim}^{\mrm{V}}_I :(\simp\mc{C})^I \longrightarrow  \simp\mc{C} $$
maps $Z:I\rightarrow \simp\mc{C}$ to $\mtt{hocolim}^{\mrm{V}}_I Z = \mrm{D}(\amalg^I Z)$, the diagonal of the bisimplicial object
given by the simplicial replacement of $Z$. In degree $n$, $\mtt{hocolim}^{\mrm{V}}_I Z$
is then equal to $\coprod_{i_0\rightarrow\cdots\rightarrow i_n} Z_{i_0,n}$.
\end{defi}

\noindent Note that if $\mtt{colim}_I Z$ exists, there is a natural augmentation $\mtt{hocolim}^{\mrm{V}}_I Z\rightarrow \mtt{colim}_I Z$ 
given by (\ref{Augmcolim}).

\begin{lema}\label{hocolimRelative} Let $(\simp\mc{C},\mc{S})$ be a relative category closed by coproducts and such that $\mc{S}$ is $\Dl$-closed. Then
the Voevodsky homotopy colimit is a relative functor. 
\end{lema}

\begin{proof}
Let $f:Z\rightarrow T$ be a morphism between the diagrams $Z,T:I\rightarrow\simp\mc{C}$ such that $f_i\in\mc{S}$ for all $i\in I$. Since $\mc{S}$ is closed by 
coproducts, $\amalg^I f : \amalg^I Z\rightarrow \amalg^I T$ is a bisimplicial morphism with $(\amalg^I f)_{n,\cdot}\in\mc{S}$ for each $n\geq 0$. Then,
$\mtt{hocolim}^{\mrm{V}}_{I} f = \mrm{D} (\amalg^I f) \in\mc{S}$ because $\mc{S}$ is $\Dl$-closed.
\end{proof}

\begin{defi}\label{hLKanDl}
Given a functor $f:I\rightarrow J$ between small categories, we def{i}ne
$f_{!}:\simp\mc{C}^I\longrightarrow\simp\mc{C}^J$ as
$$ (f_{!} X ) (j) = \mtt{hocolim}^{\mrm{V}}_{(f/j)} u_j^\ast X $$
where $u_j : (f/j)\rightarrow I$ maps $\{f(i)\rightarrow j\}$ to $i$. Note that if $\pi:I\rightarrow [0]$ is the trivial functor, then
$\pi_{!} X = \mtt{hocolim}^{\mrm{V}}_{I}  X$.
\end{defi}

Next we state the main result of the section.

\begin{thm}\label{hocolimDlClosed}
Let $(\simp\mc{C},\mc{S})$ be a relative category closed by coproducts and such that $\mc{S}$ is $\Dl$-closed. Then, the following properties hold:\\
\textbf{i.} For each small category $I$, $\mtt{hocolim}^{\mrm{V}}_I$ is a realizable homotopy colimit on
$(\simp\mc{C},\mc{S})$, invariant under homotopy right cof{i}nal changes of diagrams.\\
\textbf{ii.} For each functor $f:I\rightarrow J$ between small categories, consider $f_{!}
:\simp\mc{C}^I\rightarrow\simp\mc{C}^J$ given in \ref{hLKanDl}. Then
$(f_{!},f^\ast)$ is a relative adjoint pair between $(\simp\mc{C}^I,\mc{S})$ and $(\simp\mc{C}^J,\mc{S})$.
\end{thm}

The proof needs various auxiliary results, and will f{i}nished at the end of the section. 
We begin by investigating further properties verif{i}ed by a $\Dl$-closed class which is in addition closed by 
coproducts.

\begin{lema}\label{exseqABC} Let $(\simp\mc{C},\mc{S})$ be a relative category closed by coproducts, and such that $\mc{S}$ is $\Dl$-closed. If 
$$X^0 \stackrel{a^0}{\rightarrowtail} X^1 \stackrel{a^1}{\rightarrowtail} \cdots \stackrel{\,\,a^{m-1}}{\rightarrowtail \,\,\,} X^m \stackrel{a^{m}}{\rightarrowtail} \cdots $$
is a countable sequence of termwise coprojections in $\simp\mc{C}$, then $\mtt{colim}_m X^m$ exists. If in addition each $a^m\in\mc{S}$,
then the canonical morphism $X^0\rightarrow \mtt{colim}_m X^m$ is in $\mc{S}$ as well.
\end{lema}

\begin{proof} Assume given coprojections $a^m: X^m {\rightarrowtail} X^{m+1}$ for each $m\geq 0$. Being $\mc{C}$ closed by coproducts, it holds that 
$\mtt{colim}_m X^m$ exists in $\simp\mc{C}$ because it exists degreewise. Assume moreover that $a^{m}\in\mc{S}$ for all $m\geq 0$, and let us check that 
$X^0\rightarrow \mtt{colim}_m X^m $ is in $\mc{S}$ as well. Denote by $\mathbb{N}$ the category given by the poset of natural numbers, and consider the 
commutative square
$$\xymatrix@M=4pt@H=4pt@C=25pt{ \mtt{hocolim}^{\mrm{V}}_{\mathbb{N}} X^0 = X^0\otimes \mrm{N}(\mathbb{N})\ar[r]^-{\alpha'}\ar[d]_{f} & 
\mtt{hocolim}^{\mrm{V}}_{\mathbb{N}} X^{\bullet} \ar[d]^{g} \\ 
X^0 \ar[r]^-{\alpha} & \mtt{colim}_{\mathbb{N}} X^{\bullet}}$$
where $f$ and $g$ are the natural morphisms from the Voevodsky homotopy colimit to the colimit of a diagram, while $\alpha$ and $\alpha'$ are induced by the morphism 
of diagrams $\Lambda:X^0\rightarrow X^{\bullet}$ with $\Lambda^m = a_{m-1}\cdots a_{0}:X^0\rightarrow X^m$.
We have that $f\in\mc{S}$. Indeed, $f=X^0\otimes \pi : X^0\otimes\mrm{N}(\mbb{N}) \rightarrow X^0$. Since $\mbb{N}$ has an initial object,
$\pi:\mrm{N}(\mbb{N})\rightarrow \Dl[0]$ is a simplicial homotopy equivalence. Hence so is $f$, and in particular $f\in\mc{S}$. On the other hand,
since $\Lambda^m = a_{m-1}\cdots a_{0}\in \mc{S}$ for all $m$, it follows from Lemma \ref{hocolimRelative} that $\alpha'=\mtt{hocolim}(\Lambda)$ in $\mc{S}$.
To f{i}nish, it remains to prove that $g\in\mc{S}$.\\
Note that $g$ is the diagonal of the bisimplicial morphism $G$ given in bidegree $(n,m)$ by 
$G_{n,m}: (\mtt{hocolim}^{\mrm{V}}_{\mathbb{N}} X^{\bullet}_n)_m \rightarrow colim_{\mathbb{N}} X^{\bullet}_n$. Then it suf{f}{i}ces to prove that
$G_{n,\cdot}$ is in $\mc{S}$ for each $n\geq 0$. In other words, we may assume that the simplicial object $X^k$
is constant for each $k\geq 0$. In this case, a termwise coprojection $a^k : X^k\rightarrow X^{k+1}$ is a coprojection. 
Therefore we assume that $X^{k} = A^0\amalg \cdots \amalg A^{k}$, and
that our sequence $\{X^k,a^k\}_{k\geq 0}$ is
$$A^0\rightarrow A^0\amalg A^1 \rightarrow \cdots \rightarrow \coprod_{i=0,\cdots,k} A^i \rightarrow \cdots$$
Therefore $colim_k X^k =\coprod_{l\geq 0} A^l$, and the sequence $\{X^k,a^k\}$ is the coproduct $\amalg_{l\geq 0}\tau_l A$ of the sequences
$$\tau_l A : \ \ \ 0\rightarrow \stackrel{l)}{\cdots} \rightarrow 0 \rightarrow A^l \rightarrow A^l\rightarrow A^l\rightarrow \cdots$$
Since by def{i}nition $\mtt{hocolim}_I^{\mrm{V}}$ commutes with coproducts, 
$$\mtt{hocolim}_{\mbb{N}}^{\mrm{V}} X^{\bullet} = \coprod_{l\geq 0} \mtt{hocolim}_{\mbb{N}}^{\mrm{V}} \tau_l A = \coprod_{l\geq 0} A^l\otimes \mrm{N}(l/\mbb{N}) $$ 
Hence $g=\coprod_{l\geq 0} g^{(l)}$, where $g^{(l)}:A^l\otimes \mrm{N}(l/\mbb{N})\rightarrow A^l$ is induced by $\mrm{N}(l/\mbb{N})\rightarrow\Dl[0]$.
Then $g$ is in $\mc{S}$ because it is a coproduct of simplicial homotopy equivalences.
\end{proof}

\begin{obs} The previous results imply that
$(\simp\mc{C},\mc{S})$ is under the previous assumptions an ABC cof{i}bration
category in the sense of \cite{RB}. As proved in loc. cit. this guarantees the
existence of homotopy colimits on $(\simp\mc{C},\mc{S})$. In Theorem \ref{hocolimDlClosed} we will see in addition that 
these homotopy colimits may be computed using Voevodsky's formula.
\end{obs}

\begin{cor}\label{wheDlCl}
Let $(\simp\mc{C},\mc{S})$ be a relative category closed by coproducts, and such that $\mc{S}$ is $\Dl$-closed. If $f:L\rightarrow K$ is a map in $\simp Set$ which is a weak homotopy 
equivalence, then for each simplicial object $X$ it holds that $X\otimes f:X\otimes L\rightarrow X\otimes K$ is in $\mc{S}$.
\end{cor}

\begin{proof} The corollary is proved using a standard argument based on the $Ex_{\infty}$ f{i}brant replacement of simplicial sets and anodyne extensions.
Let $f:L\rightarrow K$ be a weak homotopy equivalence of simplicial sets, 
and $X$ be a simplicial object of $\mc{C}$. The usual resolution functor $Ex_{\infty}$ provides weak homotopy equivalences $\epsilon_L:L\rightarrow Ex_{\infty}(L)$ 
and $\epsilon_K:K\rightarrow Ex_{\infty}(K)$ such that $Ex_{\infty}(L)$ and $Ex_{\infty}(K)$ are f{i}brant simplicial sets. In addition, $\epsilon_L$ is moreover an 
anodyne extension (see \cite[p. 68]{GZ}). Since $Ex_{\infty}(f)$ is a weak homotopy equivalence between f{i}brant simplicial sets, it is a simplicial homotopy 
equivalence. Then $X\otimes Ex_{\infty}(f)\in\mc{S}$, and it suf{f}{i}ces to prove that $X\otimes \epsilon_K: X\otimes K\rightarrow X\otimes Ex_{\infty}(K)$ and 
$X\otimes \epsilon_L: X\otimes L\rightarrow X\otimes Ex_{\infty}(L)$ are in $\mc{S}$.
Therefore, it suf{f}{i}ces to prove the corollary for an anodyne extension $f$. Recall that the class of anodyne extensions is the smallest class of 
inclusions of simplicial sets containing the horn-f{i}llers $i_{k,n}:\Lambda^{k}[n]\rightarrow \Dl[n]$ for $0\leq k\leq n$, and such that it is closed by retracts, 
cobase change, small coproducts and sequential colimits. Let $\mc{W}'$ be the class of morphisms consisting of the inclusions $f:K\rightarrow L$ of simplicial sets
such that the resulting termwise coprojection $X\otimes f$ of $\simp\mc{C}$ is in $\mc{S}$ for each simplicial object $X$. Since the horn-f{i}llers $i_{k,n}$ are simplicial 
homotopy equivalences, they are in $\mc{W}'$. In addition, $\mc{W}'$ is closed by retracts and small coproducts because $\mc{S}$ is. On the other hand, 
$\mc{W}'$ is closed by cobase change and sequential colimits because trivial termwise coprojections in $(\simp\mc{C},\mc{S})$ are closed by them by Proposition 
\ref{DlClBrown} and Lemma \ref{exseqABC}.
\end{proof}

\begin{obs} Under the previous assumptions it holds moreover that $(\simp\mc{C},\mc{S})$ is a \textit{simplicial} Brown category of cof{i}brant objects. 
By this we mean that
given a cof{i}bration $f:X\rightarrow Y$ of $\simp\mc{C}$ and an inclusion
of simplicial sets $i:K\rightarrow L$, the natural morphism
$$X\otimes L \cup_{X\otimes K} Y\otimes K \longrightarrow Y\otimes L$$
is a cof{i}bration. In addition, it is in $\mc{S}$ whenever $f\in\mc{S}$
or $i$ is a weak homotopy equivalence.
\end{obs}
\vspace{0.1cm}
\subsection{The two-sided bar construction.}

\begin{defi} Given a bifunctor $F:I\times I^\comp \rightarrow \mc{D}$, the simplicial \textit{two-sided bar construction} of $F$ is the simplicial object $\mathcal{W} (F)$ given by
$$ \mathcal{W}_n (F)  =  \coprod_{i_0\rightarrow \cdots\rightarrow i_n} F(i_0,i_n) $$
The face maps of $\mathcal{W} (F)$ are def{i}ned as follows. If $0<k<n$,
 $d_k: \mathcal{W}_n (F) \rightarrow \mathcal{W}_{n-1} (F) $ sends the component $F(i_0,i_n)$ indexed by $\underline{i}=\{i_0\rightarrow \cdots\rightarrow i_n\}$ to the component $F(i_0,i_n)$ indexed by
$d_k(\underline{i})$. If $k=0$, $d_0:\mathcal{W}_n (F) \rightarrow \mathcal{W}_{n-1} (F)$ sends $F(i_0,i_n)$ with index $\underline{i}$,
to $F(i_1,i_n)$ with index $d_0(\underline{i})$ through $F(i_0\rightarrow i_1, 1_{i_n}):F(i_0,i_n)\rightarrow F(i_1,i_n)$.
If $k=n$, then $d_n:\mathcal{W}_n (F) \rightarrow \mathcal{W}_{n-1} (F)$ sends $F(i_0,i_n)$ with index $\underline{i}$
to $F(i_0,i_{n-1})$ with index $d_n(\underline{i})$ through $F(1_{i_0}, i_{n-1}\rightarrow i_n):F(i_0,i_n)\rightarrow F(i_0,i_{n-1})$. On the other hand, the degeneracy maps $s_k:\mathcal{W}_n (F) \rightarrow \mathcal{W}_{n+1} (F)$
send the term $F(i_0,i_n)$ indexed by $\underline{i}$ to the same term with index $s_k(\underline{i})$.
\end{defi}

\begin{obs}\label{coendTwoSided}\mbox{}\\
\textit{1}. Note that the colimit of $\mathcal{W} (F)$, if it exists, is the coequalizer of $d_0,d_1:\mathcal{W}_1 (F) \rightarrow \mathcal{W}_0 (F)$, which agrees with the coend $\int^i F(i,i)$
of $F$ by \cite[Lemma 1.2]{W}. Therefore $\mathcal{W} (F)$ has in this case the augmentation
$$\xymatrix@M=4pt@H=4pt@C=25pt{ \int^i F(i,i) & {\;} \coprod_{i} F(i,i) \ar[l] \; \ar@/_1.5pc/[r] & {\;} \coprod_{i\rightarrow j} F(i,j) \;
\ar@<0.5ex>[l] \ar@<-0.5ex>[l] \ar@/_1.5pc/[r]
\ar@/_1pc/[r] & \; \coprod_{i\rightarrow j\rightarrow k} F(i,k)\; \ar@<0ex>[l] \ar@<1ex>[l]
\ar@<-1ex>[l] \cdots }$$

\vspace{0.2cm}

\noindent \textit{2}. If $X:I\rightarrow \mc{D}$, the simplicial replacement of $X$ def{i}ned previously agrees with the two-sided bar construction of $X':I\times I^{\comp} \rightarrow \mc{D}$ obtained as the composition of $X$ with the projection $I\times I^{\comp}\rightarrow I$.
\end{obs}

\begin{num}\label{CoendGordo} A diagram $X:I\rightarrow \mc{D}$ induces $X\otimes \mrm{N}(\cdot / I): I\times I^\comp \rightarrow \simp\mc{D}$ given by
 $$\left( X\otimes \mrm{N}(\cdot / I)\right) (i,j) = X(i) \otimes  \mrm{N}(j / I) \ $$
Therefore $\mc{W}(X\otimes \mrm{N}(\cdot / I))$ is the bisimplicial object given in bidegree $(n,m)$ by
$$\mc{W}_{n,m}(X\otimes \mrm{N}(\cdot / I)) =
\coprod_{i_0\rightarrow \cdots\rightarrow i_n}\ \coprod_{i_n\rightarrow j_0\rightarrow \cdots\rightarrow j_m} X_{i_0}$$
\end{num}
\noindent Recall that given a simplicial object $Y$, we denote by $Y\times \Dl$ and $\Dl\times Y$ the bisimplicial objects induced by $Y$ which are constant in 
the second and f{i}rst degree respectively. Consider the natural augmentation $\beta : \mc{W}(X\otimes \mrm{N}(\cdot / I)) \rightarrow \Dl\times \amalg^I X $ 
given in bidegree $(n,m)$ by the coprojection
$$ \beta_{n,m} : \ds\coprod_{\xymatrix@R=5pt@C=0pt@H=1pt@M=0.5pt{ \scr{i_0} & \scr{\rightarrow} & \scr{\cdots} & \scr{\rightarrow} & \scr{i_n} \ar@{}|{\scr{\downarrow}}[d]\\
           \scr{j_m} &  \scr{\leftarrow} & \scr{\cdots} & \scr{\leftarrow} & \scr{j_0} }} X_{i_0} \longrightarrow  \ds\coprod_{j_0\rightarrow \cdots\rightarrow j_m} X_{j_0} $$

\vspace{0.2cm}

In this case, we have a second natural augmentation for $\mc{W}(X\otimes \mrm{N}(\cdot / I))$. Indeed, if we send $\mrm{N}(\cdot / I)$ to $\Dl[0]$, we obtain
$\alpha : \mc{W}(X\otimes \mrm{N}(\cdot / I)) \rightarrow \amalg^I X \times \Dl$ def{i}ned as the coprojection
$$
 \alpha_{n,m} : \ds\coprod_{\xymatrix@R=5pt@C=0pt@H=1pt@M=0.5pt{ \scr{i_0} & \scr{\rightarrow} & \scr{\cdots} & \scr{\rightarrow} & \scr{i_n} \ar@{}|{\scr{\downarrow}}[d]\\
           \scr{j_m} &  \scr{\leftarrow} & \scr{\cdots} & \scr{\leftarrow} & \scr{j_0} }} X_{i_0} \longrightarrow  \ds\coprod_{i_0\rightarrow \cdots\rightarrow i_n} X_{i_0}
$$

\begin{lema}\label{BarDec}
Given a diagram $X:I\rightarrow\mc{D}$, the diagonals of the bisimplicial maps $\alpha$ and $\beta$ def{i}ned above,
$\mrm{D}(\alpha),\mrm{D}(\beta) : \mrm{D} \mc{W}(X\otimes \mrm{N}(\cdot / I)) \rightarrow \amalg^I X$, are simplicial homotopy equivalences in $\simp\mc{D}$ which 
are in addition simplicially homotopic.
\end{lema}

\begin{proof}
The result is a consequence of the fact that the bisimplicial object $\mathcal{W} (X \otimes N(\cdot / I))$ is Illusie's bisimplicial decalage of
$\amalg^I X$. In bidegree $(n,m)$, $\mathcal{W}_{n,m} (X \otimes N(\cdot / I))$ is canonically isomorphic to $dec(\amalg^I X)_{n,m} = \coprod_{l_0\rightarrow \cdots\rightarrow l_{n+m+1}} X_{l_0}$,
in such a way that $\Lambda^I$ correspond to $\beta$ and $\Lambda^{II}$ correspond to $\alpha$. Therefore, the statement follows from Proposition \ref{IlDec}.
\end{proof}

\begin{num}\label{AuxCof{i}nal} Given a functor $f:I\rightarrow J$ and a diagram $X:J\rightarrow \mc{D}$,
consider the bifunctor $X\otimes \mrm{N}(\cdot / f) : J\times J^{\comp}\rightarrow \simp\mc{D}$ with
$\left( X\otimes \mrm{N}(\cdot / f)\right)(j,j') = X_j \otimes \mrm{N}(j' / f)$. As before, the associated two-sided bar construction 
$\mc{W}(X\otimes \mrm{N}(\cdot / f))$ has two natural augmentations $\alpha' : \mc{W}(X\otimes \mrm{N}(\cdot / f)) \rightarrow \amalg^J X \times \Dl$
and $\beta':\mc{W}(X\otimes \mrm{N}(\cdot / f)) \rightarrow \Dl\times \amalg^I f^\ast X$  given respectively by
$$\begin{array}{lcr}
 \alpha'_{n,m} :\!\!\!\!\!\!\! \ds\coprod_{\xymatrix@R=5pt@C=0pt@H=1pt@M=0.5pt{ \scr{j_0} & \scr{\rightarrow} & \scr{\cdots} & \scr{\rightarrow} & \scr{j_n} \ar@{}|{\scr{\downarrow}}[d]\\
           \scr{f(i_m)} &  \scr{\leftarrow} & \scr{\cdots} & \scr{\leftarrow} & \scr{f(i_0)} }}\!\!\!\!\!\!\! X_{j_0} \longrightarrow \!\!\! \ds\coprod_{j_0\rightarrow \cdots\rightarrow j_n}\!\!\!\! X_{j_0} & &
  \beta'_{n,m} :\!\!\!\!\!\!\! \ds\coprod_{\xymatrix@R=5pt@C=0pt@H=1pt@M=0.5pt{ \scr{j_0} & \scr{\rightarrow} & \scr{\cdots} & \scr{\rightarrow} & \scr{j_n} \ar@{}|{\scr{\downarrow}}[d]\\
           \scr{f(i_m)} &  \scr{\leftarrow} & \scr{\cdots} & \scr{\leftarrow} & \scr{f(i_0)} }}\!\!\!\!\!\!\! X_{j_0} \longrightarrow \!\!\! \ds\coprod_{i_0\rightarrow \cdots\rightarrow i_m}\!\!\!\! X_{f(i_0)}
  \end{array}
$$
Analogously, $\alpha'$ is induced by $\mrm{N}(\cdot / f)\rightarrow \Dl[0]$, while $\beta'_{0,\ast}$ is the coequalizer of $\mc{W}_{0,\cdot} (X\otimes \mrm{N}(\cdot / f))
\leftleftarrows \mc{W}_{1,\cdot} (X\otimes \mrm{N}(\cdot / f))$.
\end{num}

Given a diagram $Z:J\rightarrow \simp\mc{C}$ and $f:I\rightarrow J$, consider the natural map 
$$\mtt{hocolim}^{\mrm{V}}_{I} f^\ast Z = \mrm{D}\amalg^I f^\ast Z \longrightarrow \mtt{hocolim}^{\mrm{V}}_{J} Z = \mrm{D}\amalg^J Z$$ 
given by the diagonal of (\ref{SRnatural}) induced by $(f:I\rightarrow J, 1_{f^{\ast}Z}): f^\ast Z\rightarrow Z$.

\begin{prop}\label{hlcDl}
Let $(\simp\mc{C},\mc{S})$ be a relative category closed by coproducts, and such that $\mc{S}$ is $\Dl$-closed.
If $f:I\rightarrow J$ is a homotopy right cof{i}nal functor, then for each 
diagram $Z : J\rightarrow\simp\mc{C}$ the induced map $\mtt{hocolim}^{\mrm{V}}_{I} f^\ast Z\rightarrow \mtt{hocolim}^{\mrm{V}}_{J} Z$ is in $\mc{S}$.
\end{prop}

\begin{proof} The proof is a (dual) abstract version of the corresponding proof for simplicial sets given in \cite[XI; 9.2]{BK}. 
Let us see f{i}rst that $\Psi:\mtt{hocolim}^{\mrm{V}}_{I} f^\ast Z\rightarrow \mtt{hocolim}^{\mrm{V}}_{J} Z$ is in $\mc{S}$ if the diagram $Z$ is a simplicial 
constant $J$-diagram induced by $X:J\rightarrow \mc{C}$. Recall that, then, $\Psi:\amalg^I f^\ast X \rightarrow 
\amalg^J X$ is in degree $n$ the coprojection
$\coprod_{i_0\rightarrow \cdots\rightarrow i_n} X_{f(i_0)} \rightarrow \coprod_{j_0\rightarrow \cdots\rightarrow j_n} X_{j_0}$, which sends 
the term $X_{f(i_0)}$ with index $\underline{i}$ to the same term with index $f(\underline{i})$.\\ 
We will def{i}ne an inverse of $\Psi:\amalg^{I} f^\ast X\rightarrow \amalg^{J} X$ in
$\simp\mc{C} [\mc{S}^{-1}]$. Under the notations of (\ref{AuxCof{i}nal}), consider the zigzag
\begin{equation}\label{invPsi}
\xymatrix@R=25pt@C=28pt@H=4pt@M=4pt{
 \amalg^J X  & \mrm{D}\mc{W}(X\otimes \mrm{N}(\cdot / f)) \ar[l]_-{\mrm{D}(\alpha')}  \ar[r]^-{\mrm{D}(\beta')} & \amalg^I f^\ast X  }
\end{equation}
We claim that $\mrm{D}(\alpha')\in\mc{S}$. Indeed, for each $j\in J$, $\mrm{N}(j/f)\rightarrow \Dl[0]$ is by assumption a weak 
homotopy equivalence. It follows from corollary \ref{wheDlCl} that for each $j,j'\in J$, the induced map $X_{j}\otimes \mrm{N}(j' / f)\rightarrow X_j$ is in $\mc{S}$. 
Since $\mc{S}$  is assumed to be closed by coproducts, then for each $n\geq 0$, $\alpha'_{n,\cdot} :
\coprod_{\underline{j}} (X_{j_0}\otimes \mrm{N}(j_n / f)) \rightarrow \coprod_{\underline{j}} X_{j_0}$ is in $\mc{S}$ as well. But then $\mrm{D}(\alpha')\in\mc{S}$.\\
Therefore $\mrm{D}(\alpha')\in\mc{S}$ as claimed, so (\ref{invPsi}) gives a map $\Phi : \amalg^J X \dashrightarrow \amalg^I f^\ast  X$ in $\simp\mc{C} [\mc{S}^{-1}]$. Let us see that $\Phi$ is inverse to $\Psi$. Applying
Lemma \ref{BarDec} to the diagrams $X:J\rightarrow\mc{C}$ and $f^\ast X : I\rightarrow \mc{C}$, we deduce four maps
$\mrm{D}(\alpha),\mrm{D}(\beta): \mrm{D}\mc{W}(f^\ast X\otimes \mrm{N}(\cdot / I))\rightarrow \amalg^I f^\ast X$ and $\mrm{D}(\alpha''),\mrm{D}(\beta'') : 
\mrm{D}\mc{W}(X\otimes \mrm{N}(\cdot / J))\rightarrow \amalg^J X$
such that $\mrm{D}(\alpha),\mrm{D}(\beta),\mrm{D}(\alpha''),\mrm{D}(\beta'')\in\mc{S}$ and $\mrm{D}(\alpha)=\mrm{D}(\beta)$, $\mrm{D}(\alpha'')=\mrm{D}(\beta'')$ 
in $\simp\mc{C}[\mc{S}^{-1}]$.\\
Consider the natural functors $(i / I)\rightarrow (f(i) / f)$, $\{i\rightarrow i'\}\mapsto \{f(i)\rightarrow f(i')\}$,  and $(j/f) \rightarrow (j / J)$, 
$\{j\rightarrow f(i)\}\mapsto \{j\rightarrow f(i)\}$. They induce maps of bisimplicial sets
$$ \mc{W} ( f^\ast X\otimes \mrm{N}(\cdot / I))\longrightarrow \mc{W} ( X\otimes \mrm{N}(\cdot / f)) \longrightarrow \mc{W} ( X\otimes \mrm{N}(\cdot / J)) $$
$$
\ds\coprod_{\xymatrix@R=5pt@C=0pt@H=1pt@M=1pt{ \scr{i_0} & \scr{\rightarrow} & \scr{\cdots} & \scr{\rightarrow} & \scr{i_n} \ar@{}|{\scr{\downarrow}}[d]\\
           \scr{i'_m} &  \scr{\leftarrow} & \scr{\cdots} & \scr{\leftarrow} & \scr{i'_0} }} X_{f(i_0)} \longrightarrow
\ds\coprod_{\xymatrix@R=5pt@C=0pt@H=1pt@M=1pt{ \scr{j_0} & \scr{\rightarrow} & \scr{\cdots} & \scr{\rightarrow} & \scr{j_n} \ar@{}|{\scr{\downarrow}}[d]\\
           \scr{f(i'_m)} &  \scr{\leftarrow} & \scr{\cdots} & \scr{\leftarrow} & \scr{f(i'_0)} }} X_{j_0} \longrightarrow
\ds\coprod_{\xymatrix@R=5pt@C=0pt@H=1pt@M=1pt{ \scr{j_0} & \scr{\rightarrow} & \scr{\cdots} & \scr{\rightarrow} & \scr{j_n} \ar@{}|{\scr{\downarrow}}[d]\\
           \scr{j'_m} &  \scr{\leftarrow} & \scr{\cdots} & \scr{\leftarrow} & \scr{j'_0} }} X_{j_0}$$
Combining these maps with the previous augmentations we form the following diagram in which all squares commute in $\simp\mc{C}$
\begin{equation}\label{4cuadr}\xymatrix@R=23pt@C=25pt@H=4pt@M=4pt{
 \amalg^I f^\ast X \ar[r]^{1} & \amalg^I f^\ast X \ar[r]^{\Psi} & \amalg^J X \\
  \mrm{D}\mc{W} ( f^\ast X\otimes \mrm{N}(\cdot / I)) \ar[r] \ar[d]_{\mrm{D}(\alpha)} \ar[u]^{\mrm{D}(\beta)} & \mrm{D}\mc{W} ( X\otimes \mrm{N}(\cdot / f)) \ar[r] 
\ar[d]_{\mrm{D}(\alpha')} \ar[u]^{\mrm{D}(\beta')}
 & \mrm{D}\mc{W} ( X\otimes \mrm{N}(\cdot / J)) \ar[d]^{\mrm{D}(\alpha'')} \ar[u]_{\mrm{D}(\beta'')} \\
 \amalg^I f^\ast X \ar[r]^{\Psi} & \amalg^J  X \ar[r]^{1}   & \amalg^J  X  }\end{equation}
Therefore, $\Psi \, \Phi = \mrm{D}(\beta'')\, \mrm{D}(\alpha'')^{-1} = 1$ and $\Phi \, \Psi = \mrm{D}(\beta)\, \mrm{D}(\alpha)^{-1} = 1$ in $\simp\mc{C}[\mc{S}^{-1}]$.\\
Assume now that $Z:J\rightarrow \simp\mc{C}$ is a non-necessarily constant $J$-diagram. By the previous case, we deduce that for each $n\geq 0$ the simplicial map 
$\Psi'_{\cdot,n}:\amalg^{I} f^\ast Z_{\cdot,n}\rightarrow \amalg_{J} Z_{\cdot,n}$ is in $\mc{S}$. As $n$ varies $\Psi'$ is a map of bisimplicial objects whose 
diagonal agrees with $\Psi$. Therefore $\Psi$ is in $\mc{S}$ as well.
\end{proof}

As a corollary we obtain another proof for Quillen's Theorem A (cf. \cite{Q}).

\begin{cor}
If $f:I\rightarrow J$ is homotopy right cof{i}nal, then $\mrm{N}(f):\mrm{N}(I)\rightarrow\mrm{N}(J)$ is a weak homotopy equivalence of simplicial sets.
\end{cor}

\begin{proof} Weak homotopy equivalences are closed by coproducts, and form a $\Dl$-closed class of $\simp Set$ (see \cite{V}). Denote by 
$T:J\rightarrow \simp Set$ the constant diagram $j\mapsto \Dl[0]$.
By the previous theorem  $\amalg^I f^\ast T \rightarrow \amalg^J T$ is a weak homotopy equivalence. But $\amalg^I f^\ast T \cong  \mrm{N}(I)$,
$\amalg^J T \cong  \mrm{N}(J)$ and the previous map corresponds to $\mrm{N}(f)$, which is then a weak homotopy equivalence.
\end{proof}

\begin{proof}[\textbf{Proof of Theorem \ref{hocolimDlClosed}}]
We have seen in Lemma \ref{hocolimRelative} that $\mtt{hocolim}^{\mrm{V}}$ is a relative functor, then so is $f_{!}$ by def{i}nition. 
The fact that $\mtt{hocolim}_I^{\mrm{V}}$ is a realizable homotopy colimit is deduced from \textit{ii} putting $J=[0]$ and 
$f:I\rightarrow [0]$ the trivial functor.\\ 
Let us see then part \textit{ii}. So let $f:I\rightarrow J$
be a functor between small categories. Given $Y:J\rightarrow\simp\mc{C}$, $j\in J$ and $n\geq 0$, the adjunction morphism $\theta^f_Y : f_{!} f^\ast Y\rightarrow Y$ is
$$ (\theta^f_Y)_n (j): (f_{!} f^\ast Y)_n (j) = \ds\coprod_{f(i_0)\rightarrow \cdots \rightarrow f(i_n)\rightarrow j}\!\!\!\! Y_{f(i_0),n} \ \ \longrightarrow \ \  Y_{j,n} $$
induced by $Y(f(i_0)\rightarrow j) : Y_{f(i_0),n}\rightarrow Y_{j,n}$.\\
Next we describe the adjunction morphism
$\Phi : 1_{\simp\mc{C}^I}\dashrightarrow f^\ast f_{!} $ in $Fun(\simp\mc{C}^I,\simp\mc{C}^I)[\mc{S}^{-1}]$.
To this end we use the `replacement' $Q: \simp\mc{C}^I\rightarrow \simp\mc{C}^I$ of \cite[8.2]{Du}.
By def{i}nition,  $Q:= {1_I}_{!}$ is the homotopy left Kan extension associated with $1_I:I\rightarrow I$.
If $X:I\rightarrow\simp\mc{C}$, then
$$ ( Q X )(i) = \mtt{hocolim}^{\mrm{V}}_{(I/i)} u_i^\ast X$$
where $u_i : (I/i) \rightarrow I$ is def{i}ned as $u_i (\{i'\rightarrow i\}) = i'$. Note that $u_i$ factors as $w_i:(f/f(i))\rightarrow I$ followed by $v_i:(I/i)\rightarrow (f/f(i))$ , where $w_i(\{f(i')\rightarrow f(i)\})=i'$ and $v_i(\{i'\rightarrow i\})=\{f(i')\rightarrow f(i)\}$.
Then, we have an induced map
$\psi_X(i) : (QX)(i)\rightarrow \mtt{hocolim}^{\mrm{V}}_{(f/f(i))} w_i^\ast X = (f_{!} X)(f(i))= (f^\ast f_{!} X) (i)$.
It gives rise to a natural transformation $\psi : Q\rightarrow f^\ast f_{!}$.\\
Assume constructed a natural transformation $\rho:Q\rightarrow 1_{\simp\mc{C}^I}$ such that $(\rho_X)(i)\in\mc{S}$ for each diagram $X$ and $i\in I$. We then 
def{i}ne $\Phi : 1_{\simp\mc{C}^I}\dashrightarrow f^\ast f_{!} $ as $\Phi_X=\psi_X \, \rho_X^{-1}$. That is, as the zigzag
$$\xymatrix@M=4pt@H=4pt@C=35pt{ X &  QX \ar[l]_-{\rho_X} \ar[r]^-{\psi_X} & f^\ast f_{!} X}$$
Since each $(I / i)$ has as f{i}nal object $1_i : i\rightarrow i$, we have a natural transformation
$\tau : u_i^\ast X \rightarrow  \pi^\ast X_i $ f{i}lling the triangle
$$\xymatrix@M=4pt@H=4pt@C=35pt{
 (I / i)  \ar[r]^-{\pi} \ar[rd]_{u_i^\ast X}^{{}^{\stackrel{\tau}{\Rightarrow}}} &  [0]  \ar[d]^{X_i}\\
                                      & \simp\mc{C} }$$
Given $\{f:i'\rightarrow i\}\in (I/i)$, $\tau_f$ is def{i}ned as $X(f):X_{i'} \rightarrow X_i$. We obtain $\rho_X(i)=\mtt{hocolim}^{\mrm{V}}(\pi,\tau):\mtt{hocolim}^{\mrm{V}}_{(I/i)} u_i^\ast X \rightarrow  \mtt{hocolim}^{\mrm{V}}_{[0]} X_i = X_i$. It is straightforward to check that $\rho_X(i)$ is natural on $i$ and $X$, so $\rho:Q\rightarrow 1_{\simp\mc{C}}$ is a natural transformation.\\
Let us see now that $\rho_X(i)\in\mc{S}$ for all $i$ in $I$. Consider each $i\in I$ as the functor $i : [0]\rightarrow (I/i)$, $0\mapsto 1_i$.  This is a
homotopy right cof{i}nal functor since $1_i$ is a f{i}nal object of $(I/i)$. It follows from Proposition \ref{hlcDl} that the induced map
$\mtt{hocolim_{[0]}} i^\ast u_i^\ast X = X_i \rightarrow \mtt{hocolim}^{\mrm{V}}_{(I/i)} u_i^\ast X = (QX) (i)$ is in $\mc{S}$. On the other hand, the composition of
this map with $(Q X)(i)\rightarrow X_i$ is clearly the identity on $X_i$. Hence $\rho_X(i)\in\mc{S}$ as required.
To f{i}nish, it remains to see that the compositions
\begin{eqnarray}
\xymatrix@M=4pt@H=4pt@C=35pt{ f^\ast   \ar[r]^-{\Phi_{f^\ast}} & f^\ast f_{!} f^\ast  \ar[r]^-{f^\ast \theta^f} & f^\ast}\label{Adj1Dl}\\
\xymatrix@M=4pt@H=4pt@C=42pt{  f_{!}\ar[r]^-{f_{!} \Phi} & f_{!} f^\ast f_{!}  \ar[r]^-{\theta^f_{f_{!} }} & f_{!}  }\label{Adj2Dl}
\end{eqnarray}
are the identity in $Fun(\simp\mc{C}^J,\simp\mc{C}^I)[\mc{S}^{-1}]$ and $Fun(\simp\mc{C}^I,\simp\mc{C}^J)[\mc{S}^{-1}]$, respectively.
It is easy to see that (\ref{Adj1Dl}) is the identity in $Fun(\simp\mc{C}^J,\simp\mc{C}^I)[\mc{S}^{-1}]$. Indeed, it follows directly from the def{i}nitions
that $\rho_{f^\ast Y} : Q(f^\ast Y) \rightarrow f^\ast Y$ agrees with $\theta_Y^f \, \psi_{f^\ast Y}$ in $\simp\mc{C}^I$, for each $Y:J\rightarrow \simp\mc{C}$. 
The fact that (\ref{Adj2Dl}) is equal to the identity is more involved. We must see that 
$ a=f_{!} \rho :  f_{!} Q \rightarrow  f_{!} $ agrees with $b=\theta^f_{ f_{!} } \,  f_{!} \psi$ in $Fun(\simp\mc{C}^I,\simp\mc{C}^J)[\mc{S}^{-1}]$. 
If $n\geq 0$, $j\in J$ and $X:I\rightarrow\simp\mc{C}$, then $f_{!} QX:J\rightarrow\simp\mc{C}$ is given by
$$(f_{!} QX)_n(j) =\!\!\coprod_{f(i_0)\rightarrow \cdots\rightarrow f(i_n)\rightarrow j}\ \ \coprod_{l_0\rightarrow\cdots\rightarrow l_n\rightarrow i_0} X_{l_0,n} = \!\!
\ds\coprod_{\xymatrix@R=5pt@C=0pt@H=1pt@M=0.5pt{ & \mbox{} &  \scr{f(l_0)} & \scr{\rightarrow} & \scr{\cdots} & \scr{\rightarrow} & \scr{f(l_n)} \ar@{}|{\scr{\downarrow}}[d]\\
     \scr{j} & \scr{\leftarrow}   &    \scr{f(i_n)} &  \scr{\leftarrow} & \scr{\cdots} & \scr{\leftarrow} & \scr{f(i_0)} }} X_{l_0,n}$$
Note that $f_{!} Q$ is by def{i}nition the composition of $\widetilde{F}:\simp\mc{C}^I\rightarrow \simp\simp\mc{C}^J$ with the diagonal $\mrm{D}: \simp\simp\mc{C}^J\rightarrow \simp\mc{C}^J$, where
$$(\widetilde{F} X)_{n,m}(j) = \ds\coprod_{\xymatrix@R=5pt@C=0pt@H=1pt@M=0.5pt{ & \mbox{} &  \scr{f(l_0)} & \scr{\rightarrow} & \scr{\cdots} & \scr{\rightarrow} & \scr{f(l_n)} \ar@{}|{\scr{\downarrow}}[d]\\
     \scr{j} & \scr{\leftarrow}   &    \scr{f(i_n)} &  \scr{\leftarrow} & \scr{\cdots} & \scr{\leftarrow} & \scr{f(i_0)} }} X_{l_0,m}$$
Also $f_{!}$ is by def{i}nition the composition of ${F}:\simp\mc{C}^I\rightarrow \simp\simp\mc{C}^J$ with the diagonal, where 
$$(F\, X)_{n,m}(j) = \ds\coprod_{\scr{ f(i_0)\rightarrow\cdots\rightarrow f(i_n)}\rightarrow j}  X_{i_0 , m}
$$
On the other hand $a$ and $b$ are the diagonals of the bisimplicial maps $a'$ and $b'$ def{i}ned as
$$ (a'_X)_{n,m}(j) :\!\!\! \ds\coprod_{\xymatrix@R=5pt@C=0pt@H=1pt@M=0.5pt{ & \mbox{} &  \scr{f(l_0)} & \scr{\rightarrow} & \scr{\cdots} & \scr{\rightarrow} & \scr{f(l_n)} \ar@{}|{\scr{\downarrow}}[d]\\
     \scr{j} & \scr{\leftarrow}   &    \scr{f(i_n)} &  \scr{\leftarrow} & \scr{\cdots} & \scr{\leftarrow} & \scr{f(i_0)} }} X_{l_0,m} \longrightarrow
\ds\coprod_{\xymatrix@R=5pt@C=0pt@H=1pt@M=0.5pt{ \scr{f(i_0)} & \scr{\rightarrow} & \scr{\cdots} & \scr{\rightarrow} & \scr{f(i_n)} &\scr{\rightarrow}     \scr{j}}} X_{i_0,m}$$
$$(b'_X)_{n,m}(j):\!\!\! \ds\coprod_{\xymatrix@R=5pt@C=0pt@H=1pt@M=0.5pt{ & \mbox{} &  \scr{f(l_0)} & \scr{\rightarrow} & \scr{\cdots} & \scr{\rightarrow} & \scr{f(l_n)} \ar@{}|{\scr{\downarrow}}[d]\\
     \scr{j} & \scr{\leftarrow}   &    \scr{f(i_n)} &  \scr{\leftarrow} & \scr{\cdots} & \scr{\leftarrow} & \scr{f(i_0)} }} X_{l_0,m}  \longrightarrow
\ds\coprod_{\xymatrix@R=5pt@C=0pt@H=1pt@M=0.5pt{ \scr{f(l_0)} & \scr{\rightarrow} & \scr{\cdots} & \scr{\rightarrow} & \scr{f(l_n)} &\scr{\rightarrow}     \scr{j}}} X_{l_0,m}
$$
Putting all together, we have $a',b':\widetilde{F}\rightarrow F$ in $Fun(\simp\mc{C}^I,\simp\simp\mc{C}^J)$ such that $\mrm{D}a' = a$ and $\mrm{D} b' = b$.
Note that for f{i}xed $m\geq 0$ and $j\in J$, it holds that
$(\widetilde{F}X)_{\cdot,m} (j)$ agrees with $\mrm{D}(dec(\amalg^{(f/j)} u_j^{\ast} X_{\cdot,m}))$, the diagonal of Illusie's decalage of $\amalg^{(f/j)} u_j^\ast X_{\cdot,m}$.
In addition, $a'_{\cdot,m}(j)$ and $b'_{\cdot,m}(j)$ correspond to the natural augmentations $\mrm{D}(\Lambda^I),\mrm{D}(\Lambda^{II}):\mrm{D}( 
dec(\amalg^{(f/j)} u_j^\ast X_{\cdot,m}))\rightarrow \amalg^{(f/j)} u_j^\ast X_{\cdot,m}$ associated with $dec(\amalg^{(f/j)} u_j^\ast X_{\cdot,m})$.
Therefore by Proposition \ref{IlDec} we have a simplicial homotopy $H_{\cdot,m}(j): (\widetilde{F}X)_{\cdot,m} (j)\otimes \Dl[1]\rightarrow 
({F}X)_{\cdot,m} (j)$ between $a'_{\cdot,m}(j)$ and $b'_{\cdot,m}(j)$. Since the homotopies $H_{\cdot,m}(j)$ are natural they produce as $m$ and $j$ vary a 
natural bisimplicial map $H:\widetilde{F}X\otimes \Dl[1]\rightarrow FX$, which gives indeed a natural transformation
$H:\widetilde{F}-\otimes \Dl[1]\rightarrow F$ with $H\, d_0 = a'$ and $H\, d_1 = b'$. 
Then, taking diagonals we get $h=\mrm{D}H : \mrm{D}(\widetilde{F}-\otimes \Dl[1])=(\mrm{D}\widetilde{F})\otimes \Dl[1]\rightarrow \mrm{D}F$ such that
$h\, d_0 = a$ and $h\, d_1 = b$.
But  $d_0$ and $d_1$ are pointwise in $\mc{S}$ because they are simplicial homotopy equivalences, and have also a common section, so $d_0=d_1$ in 
$Fun(\simp\mc{C}^I, \simp\mc{C}^J)[\mc{S}^{-1}]$. We then conclude that $a=b$ there as well.\\
To f{i}nish the proof, we must check that $\mtt{hocolim}_I^{\mrm{V}}$ is invariant under homotopy right cof{i}nal changes of diagrams.
Given a homotopy right cof{i}nal functor $f:I\rightarrow J$, we claim that the relative natural transformation 
$\mtt{hocolim}^{\mrm{V}}(f): \mtt{hocolim}_I^{\mrm{V}}f^{\ast} \dashrightarrow \mtt{hocolim}_J^{\mrm{V}}$ induced by adjunction as in 
(\ref{hrcdef{i}bis}) agrees in $Fun(\simp\mc{C}^J, \simp\mc{C})[\mc{S}^{-1}]$
with the natural transformation $\Psi: \mtt{hocolim}_I^{\mrm{V}}f^{\ast} \rightarrow \mtt{hocolim}_J^{\mrm{V}}$ of Proposition 
\ref{hlcDl}. Then $\mtt{hocolim}^{\mrm{V}}(f)$ would be an isomorphism of $\mc{R}el\mc{C}at$. Let us see the claim. Using the explicit adjunction morphisms 
described before, $\mtt{hocolim}^{\mrm{V}}(f)$ is the composition of 
$(\mtt{hocolim}^{\mrm{V}}_I f^{\ast}\compc\rho)^{-1}: \mtt{hocolim}^{\mrm{V}}_I f^{\ast}\rightarrow \mtt{hocolim}^{\mrm{V}}_I f^{\ast} Q $ with 
$$\xymatrix@M=4pt@H=4pt@C=47pt{  
 \mtt{hocolim}^{\mrm{V}}_I f^{\ast} Q 
\ar[r]^-{\mtt{hocolim}^{\mrm{V}}_I f^{\ast}\compc \psi} & \mbox{$\begin{array}{c}\mtt{hocolim}^{\mrm{V}}_I f^{\ast} c_J  \mtt{hocolim}^{\mrm{V}}_J\\[0.1cm] 
        =  \mtt{hocolim}^{\mrm{V}}_I c_I  \mtt{hocolim}^{\mrm{V}}_J \end{array}$} \ar[r]^-{\theta\compc \mtt{hocolim}^{\mrm{V}}_J} 
&  \mtt{hocolim}^{\mrm{V}}_J } $$  
Note that $\mtt{hocolim}^{\mrm{V}}_I f^{\ast} Q X$ is, by def{i}nition, equal to $\mrm{D}\mc{W} ( X\otimes \mrm{N}(\cdot / f)) $, the diagonal of the 
two-sided bar construction associated with $ X\otimes \mrm{N}(\cdot / f): J\times J^{\comp} \rightarrow \simp\mc{C}$. Under the notations of 
(\ref{AuxCof{i}nal}), $\mtt{hocolim}^{\mrm{V}}_I f^{\ast}\compc\rho$ is just the diagonal of $\beta'$, while 
$(\theta\compc \mtt{hocolim}^{\mrm{V}}_J)\, (\mtt{hocolim}^{\mrm{V}}_I f^{\ast}\compc \psi)$ is just the diagonal of 
$\alpha'$.\\ 
Then, to see that $\mtt{hocolim}^{\mrm{V}}(f)$ agrees with $\Psi$ in  $Fun(\simp\mc{C}^J, \simp\mc{C})[\mc{S}^{-1}]$ it suf{f}{i}ces to check 
that $ \Psi \, \mrm{D}\beta'$ is equal to  $\mrm{D}\alpha'$ there. But this is proved easily using the right hand side of diagram (\ref{4cuadr}) 
and the natural homotopies of Lemma \ref{BarDec}.
\end{proof}


\section{Characterization of realizable homotopy colimits.}\label{sectionDescenso}

In this section we characterize those relative categories closed by coproducts and possessing all realizable homotopy colimits, that are invariant 
under homotopy right cof{i}nal changes of diagrams. We prove that these are precisely the relative categories closed by coproducts and possessing a 
`geometric realization' 
for simplicial objects, which we call \textit{simple functor}. This is encoded in the notion of simplicial descent category 
(see Def{i}nition \ref{Def{i}Descenso}). More concretely, the main result of the section is the following 

\begin{thm}\label{characterization} Let $(\mc{C},\mc{W})$ be a relative category closed by coproducts. The following are equivalent:
\begin{compactitem}
\item[\textbf{i.}] $(\mc{C},\mc{W})$ admits realizable homotopy colimits $\mtt{hocolim}_I:\mc{C}^I\rightarrow \mc{C}$, which
are invariant under homotopy right cof{i}nal changes of diagrams.
\item[\textbf{ii.}] $(\mc{C},\mc{W})$ admits a simplicial descent structure with simple functor $\mbf{s}:\simp\mc{C}\rightarrow \mc{C}$.
\end{compactitem}
In addition, if these equivalent conditions hold then for each small category $I$ there is a natural 
isomorphism $\mtt{hocolim}_I \simeq \mbf{s}\amalg^I$ of $\mc{R}el\mc{C}at$, 
where $\amalg^I:\mc{C}^I\rightarrow \simp\mc{C}$ denotes the simplicial replacement.
\end{thm}

\begin{obs}
We will see moreover in Theorem \ref{AdjointPair} that for relative categories satisfying the equivalent conditions of 
previous theorem all homotopy left Kan extension exist 
and may be computed pointwise.
\end{obs}

Before going into detail, let us explain why the formula $$\mtt{hocolim}_I \simeq \mtt{hocolim}_{\simp} \amalg^{I}$$ obtained in previous theorem holds.
Recall that in the context of colimits a similar formula works. More concretely, given a diagram $X:I\rightarrow \mc{C}$ we know that $\mtt{colim}_I X$ agrees 
with the coequalizer of $\amalg_{i\in I} X_i  \leftleftarrows \amalg_{i\rightarrow j} X_i$, which is a truncation of $\amalg^I X$.
It is straightforward to prove this equality just using the def{i}nitions, but it may be obtained as well from the following a posteriori argument.

Recall that given $f:I\rightarrow J$, a left Kan extension $f_{!}: \mc{C}^I\rightarrow \mc{C}^J$ is def{i}ned as a left adjoint
of $f^{\ast}: \mc{C}^J\rightarrow \mc{C}^I$. Note that $\mtt{colim}_J f_{!} \simeq \mtt{colim}_I$ because both are left adjoints of $c_I = f^{\ast} c_J$. 

Consider the equalizer category  $\mtt{eq}$ consisting of 
$$\xymatrix@M=4pt@H=4pt@C=33pt{ [0]\ar@<1ex>[r]^{d^0} \ar@<-1ex>[r]_{d^1} & [1]   }$$ 
Given a small category $I$ denote by $\mtt{eq}^{\circ}.I$ the category with objects the functors  $\alpha : [m]\rightarrow I$ for $m=0,1$. A 
morphism $(\alpha,[m])\rightarrow (\alpha',[m'])$ in $\mtt{eq}^{\circ}.I$ is a morphism $f:[m']\rightarrow [m]$ of $\mtt{eq}$ such that 
$\alpha f= \alpha'$. We have the functors 
$$\xymatrix@M=4pt@H=4pt@C=33pt{ I & \mtt{eq}^{\circ}.I \ar[l]_{p} \ar[r]^{q} & \mtt{eq}^{\circ}} $$
given by $p(\alpha:[m]\rightarrow I) = \alpha (0)$ and $q(\alpha:[m]\rightarrow I) = [m]$, for $m=0,1$. It is easy to see that:

\begin{compactitem}
 \item[\textit{i.}] $p$ is right cof{i}nal. Then $\mtt{colim}_I X \simeq \mtt{colim}_{\mtt{eq}^{\circ}.I}\, p^{\ast} X$.
 \item[\textit{ii.}] $q$ has a left Kan extension $q_{!}:\mc{C}^{\mtt{eq}^{\circ}.I}\rightarrow \mc{C}^{\mtt{eq}^{\circ}}$ given by 
$$(q_{!} Y)([m])= \amalg_{(\alpha,[m])\in \mtt{eq}^{\circ}.I} Y (\alpha,[m])$$
for $m=0,1$. Then $\mtt{colim}_{\mtt{eq}^{\circ}.I}\,Y \simeq   \mtt{colim}_{\mtt{eq}^{\circ}}\, q_{!} Y$.
\end{compactitem}
Therefore $\mtt{colim}_I X \simeq \mtt{colim}_{\mtt{eq}^{\circ}}\, q_{!} p^{\ast} X$. That is, $\mtt{colim}_I$
may be computed by the coequalizer of $q_{!} p^{\ast} X$, which is by def{i}nition $\{\amalg_{i\in I} X_i  \leftleftarrows \amalg_{i\rightarrow j} X_i\}$.\\

To adapt this argument to the homotopical setting, we f{i}rst need to replace $p$ with a homotopy right cof{i}nal functor. This can be done by considering $\Dl$ 
instead $\mtt{eq}$. That is, given a small category $I$ we construct its category of simplexes $\simp.I$ analogously, and we have two functors  
$$\xymatrix@M=4pt@H=4pt@C=33pt{ I & \simp.I \ar[l]_{p} \ar[r]^{q} & \simp} $$
given by $p(\alpha:[m]\rightarrow I) = \alpha (0)$ and $q(\alpha:[m]\rightarrow I) = [m]$, for $m\geq 0$. This time $p$ is homotopy right cof{i}nal. 
Our next task would be to obtain a relative left adjoint $q_{!} : (\mc{C}^{\simp.I},\mc{W})\rightarrow (\simp\mc{C},\mc{W})$ of $q^\ast:
(\simp\mc{C},\mc{W})\rightarrow (\mc{C}^{\simp.I},\mc{W}) $. We have the candidate $q_{!} : \mc{C}^{\simp.I}\rightarrow  \simp\mc{C}$ def{i}ned as 
$$(q_{!} Y)([m])= \amalg_{(\alpha,[m])\in \mtt{eq}^{\circ}.I} Y (\alpha,[m])$$ 
for $m\geq 0$. It is indeed left adjoint to 
$q^{\ast} :\simp \mc{C}\rightarrow \mc{C}^{\simp.I}$.
If we want $q_{!}:(\mc{C}^{\simp.I},\mc{W})\rightarrow (\simp\mc{C},\mc{W}) $ to be a relative functor, we must ask $\mc{W}$ to be closed by coproducts.
Under this assumption we do have that $q_{!}$ is a relative left adjoint of $q^\ast$.
 
Therefore, in case $(\mc{C},\mc{W})$ is a homotopical category closed by coproducts, we have obtained the a posteriori formula  
$$\mtt{hocolim}_I X \simeq \mtt{hocolim}_{\simp} q_{!}p^{\ast} X$$
for realizable homotopy colimits that are invariant under homotopy right cof{i}nal changes of diagrams. To f{i}nish, the simplicial object $q_{!}p^{\ast} X$ is 
nothing else than $\amalg^I X$, the simplicial replacement of $X$.
\vspace{0.2cm}
\subsection{Simplicial descent categories.}\mbox{}\\[0.2cm]
\indent
Next we recall the def{i}nition of simplicial descent categories, as well as some of their properties needed later. 
We use here a slight variant of the notion given in \cite{R}, more suitable for 
the setting of relative categories and relative adjunctions. More concretely, the only dif{f}erence is that now the 
transformations $\mu$ and $\lambda$ of (S3) and (S4) are assumed to be isomorphisms of $\mc{R}el\mc{C}at$ 
instead of zigzags of natural weak equivalences as in loc. cit. This is a minor change not af{f}ecting the main 
properties of simplicial descent categories.

\begin{defi}\label{Def{i}Descenso} Let $(\mc{C},\mc{W})$ be a relative category closed by f{i}nite coproducts. A \textit{simplicial descent structure} on $(\mc{C},\mc{W})$ is 
 a triple $(\mbf{s},\mu,\lambda)$ satisfying the following f{i}ve axioms:\\[0.1cm]
$\mathbf{(S1)}$ $\mbf{s}:\simp\mc{C}\rightarrow \mc{C}$ is a functor, called the \textit{simple functor}. It commutes with f{i}nite coproducts
up to weak equivalence. More precisely, the canonical morphism
$\mbf{s}X \amalg \mbf{s}(Y)\rightarrow \mbf{s}(X\amalg Y)$ is a weak equivalence for all $X$ and $Y$ in $\simp\mc{C}$.\\[0.05cm]
$\mathbf{(S2)}$ If $f:X\rightarrow Y$ is a morphism in $\simp\mc{C}$ such that $f_n$ is a weak equivalence for all $n\geq 0$,
then $\mbf{s}(f)$ is a weak equivalence. In other words, $\mbf{s}: (\simp\mc{C} , \mc{W})\rightarrow (\mc{C} , \mc{W})$ is a relative functor.\\[0.05cm]
$\mathbf{(S3)}$ $\mu:\mbf{s}\mrm{D}\dashrightarrow\mbf{s}\mathbf{s}$ is an isomorphism of $Fun(\simp\simp\mc{C},\mc{C})[\mc{W}^{-1}]$. If
$Z$ is a bisimplicial object, recall that $\mbf{s}\mrm{D}(Z)$ is the simple
of the diagonal of $Z$. On the other hand
$\mbf{s}\mbf{s}(Z):=\mbf{s}(n\rightarrow\mbf{s}(m\rightarrow
Z_{n,m}))$ is the iterated simple of $Z$.\\[0.05cm]
$\mathbf{(S4)}$ $\lambda:\mbf{s} c\dashrightarrow
1_{\mc{C}}$ is an isomorphism of  $Fun(\mc{C},\mc{C})[\mc{W}^{-1}]$, which is assumed to be compatible with $\mu$ in the sense of (\ref{compatibLambdaMu}) below.\\[0.05cm]
$\mathbf{(S5)}$  The image under the simple functor of the map $d_0^A:A\rightarrow A\otimes \Dl[1]$ is a weak equivalence for each object $A$ of $\mc{C}$.\\[0.05cm]
A \textit{simplicial descent category} is a relative category closed by f{i}nite coproducts and endowed with a simplicial descent structure. We will denote 
a simplicial descent category by $(\mc{C},\mc{W},\mbf{s},\mu,\lambda)$, and also by $(\mc{C},\mc{W},\mbf{s})$, or even $(\mc{C},\mc{W})$, for brevity.
\end{defi}

\begin{num}\label{compatibLambdaMu}[Compatibility between $\mu$ and $\lambda$.]
Given a simplicial object $X$, denote by $X\times\Dl$, $\Dl\times X$ the bisimplicial objects with
$(X\times\Dl)_{n,m}=X_n$ and $(\Dl\times X)_{n,m}=X_m$. Note that $\mbf{s}\mbf{s}(X\times\Dl)=
\mbf{s}(n\rightarrow \mbf{s}c(X_n))$ and
$\mbf{s}\mbf{s}(\Dl\times X)= \mbf{s}c\mbf{s}(X)$. The
compositions
\begin{equation}\label{compatibLambdaMuEquac}\xymatrix@M=4pt@H=4pt@R=7pt@C=25pt{
 \mbf{s}(X) \ar[r]^-{\mu_{\Dl\times X}} &  \mbf{s}c\mbf{s}(X)\ar[r]^-{\lambda_{\mbf{s}(X)}} & \mbf{s}(X) &
 \mbf{s}(X) \ar[r]^-{\mu_{X\times \Dl}} &  \mbf{s}\mbf{s}c(X)\ar[r]^-{\mbf{s}(\lambda_{X})} &
 \mbf{s}(X)}\end{equation}
give rise to isomorphisms of $\mbf{s}$ in
$Fun(\simp\mc{C},\mc{C})[\mc{W}^{-1}]$. Then, $\lambda$ is said to be \textit{compatible} with $\mu$ if these isomorphisms are the identity in 
$Fun(\simp\mc{C},\mc{C})[\mc{W}^{-1}]$.\\
\end{num}

We will use that simplicial descent structures are inherited by diagram categories, that is an easy consequence of the axioms.

\begin{prop}\label{DescensoFuntores} Let $I$ be a small category and let $(\mbf{s},\mu,\lambda)$ be a simplicial descent structure on $(\mc{C},\mc{W})$. Then, the triple
$(\mbf{s}^{I},\mu^{I},\lambda^{I})$
def{i}ned pointwise is a simplicial descent structure on $(\mc{C}^I,\mc{W})$. Given $X:\simp\rightarrow \mc{C}^I$,
$(\mbf{s}^{I}X)(i)=\mbf{s}(n\rightarrow X_n(i))$, and $(\mu^{I},\lambda^{I})$ is def{i}ned analogously.
\end{prop}

Simplicial descent categories are closely related to Voevodsky $\Dl$-closed classes. More concretely,

\begin{prop}\label{DlClosedDescenso}\mbox{}\\
\vspace{-0.5cm}
\begin{compactitem}
\item[\textbf{\textit{i}.}] Consider a relative category $(\simp\mc{C},\mc{S})$ closed by f{i}nite coproducts. The following are equivalent:
\begin{compactitem}
\item[\textit{1}.] $\mc{S}$ is $\Dl$-closed.
\item[\textit{2}.] $(\simp\mc{C},\mc{S},\mrm{D}:\simp\simp\mc{C}\rightarrow\simp\mc{C})$ is a simplicial descent category.
\end{compactitem}
\item[\textbf{\textit{ii}.}] Given a simplicial descent category $(\mc{C},\mc{W},\mbf{s})$  then the class of $\mc{S}=\mbf{s}^{-1}\mc{W}$ of $\simp\mc{C}$ 
is $\Dl$-closed and closed by f{i}nite coproducts.
If in addition $\mc{W}$ is closed by coproducts, then so is $\mc{S}$.
\end{compactitem}
\end{prop}

\begin{proof} The last statement of \textit{ii} follows from Lemma \ref{SimpleCoprod}. The remaining assertions of the proposition are proved in 
\cite[Theorem 4.2]{R}.
\end{proof}

\begin{lema}\label{SimpleCoprod}
Let $(\mc{C},\mc{W})$ be a relative category closed by coproducts, and endowed with a simplicial descent structure. Then the simple functor $\mbf{s}$ preserves 
all small coproducts up to weak equivalence. That is, given a family $\{X^\alpha\}_{\alpha\in\Lambda}$ of simplicial objects, the following natural map is a weak equivalence
$$ \coprod_\alpha \mbf{s}(X^\alpha) \longrightarrow \mbf{s} (\coprod_\alpha X^\alpha) \ .$$
\end{lema}

\begin{proof}
The simple functor $\mbf{s} : \simp\mc{C} [\mc{W}^{-1}]\rightarrow \mc{C} [\mc{W}^{-1}]$ is a left adjoint by Theorem \ref{equivCat}, so it commutes with colimits. Although colimits in localized categories rarely exist, the case of coproducts is an exception to this rule. Indeed, since by assumption the
class $\mc{W}$ is closed by small coproducts, it follows that $\mc{C}[\mc{W}^{-1}]$ is so, and that the localization functor $\mc{C}\rightarrow \mc{C}[\mc{W}^{-1}]$ preserves coproducts. The same applies
to $\simp\mc{C}[\mc{W}^{-1}]$. Therefore,
$\coprod_\alpha \mbf{s}(X^\alpha) \longrightarrow \mbf{s} (\coprod_\alpha X^\alpha)$ is an isomorphism in $\mc{C}[\mc{W}^{-1}]$ since $\mbf{s}$ commutes with coproducts there. But, being $\mc{W}$ saturated, this implies the result.
\end{proof}

The following result is the key point to transfer the properties of $\Dl$-closed classes to general simplicial descent categories.

\begin{thm}\label{equivCat} Let $(\mbf{s},\mu,\lambda)$ be a simplicial descent structure on $(\mc{C},\mc{W})$. Then, the following properties hold:\\
\textbf{i.} The simple functor $\mbf{s}:\simp\mc{C}\rightarrow \mc{C}$ is a realizable homotopy colimit.\\
\textbf{ii.} Set $\mc{S}=\mbf{s}^{-1}\mc{W}$. Then, $(\mbf{s},c)$ is a relative adjoint equivalence between $(\simp\mc{C},\mc{S})$ and $(\mc{C},\mc{W})$.
\end{thm}

The proof is the same as the one of \cite[Theorem 5.1]{R}. In any case, we will explicitly check \textit{ii} in the proof of next proposition, which allows to weaken
axiom (S3). This weakening will only be used in section \ref{sectionBK}, where we see that the Bousf{i}eld-Kan simplicial homotopy colimit  gives 
a simplicial descent structure  on the subcategory of cof{i}brant objects of any model category.

\begin{prop}\label{simplif{i}cacionAxiomas} Let $(\mc{C},\mc{W})$ be a relative category closed by f{i}nite coproducts. Assume that there exists $(\mbf{s},\lambda)$ as in 
def{i}nition \ref{Def{i}Descenso} satisfying axioms \emph{(S1)}, \emph{(S2)}, \emph{(S4)} and \emph{(S5)}, and
\begin{quote}
 \emph{(S3)'} Given a bisimplicial morphism $F$ of $\simp\simp\mc{C}$, it holds that $\mbf{s}(n\rightarrow \mbf{s}(m\rightarrow F_{n,m}))$ is a weak equivalence
if and only if $\mbf{s}(n\rightarrow F_{n,n})$ is so.
\end{quote}
Then $(\mc{C},\mc{W})$ is a simplicial descent category. More precisely, there exists $\mu:\mbf{s}\mrm{D}\dashrightarrow\mbf{s}\mathbf{s}$ such that 
$(\mbf{s},\mu,\lambda)$ is a simplicial descent structure on $(\mc{C},\mc{W})$.
\end{prop}

\begin{proof} Set $\mc{S}=\mbf{s}^{-1}\mc{W}$. Let us see that axioms (S1), (S2), (S3)', (S4) and (S5) still guarantee that there exists an isomorphism 
$\Phi:1_{\simp\mc{C}}\rightarrow c\mbf{s}$ of $Fun(\simp\mc{C},\simp\mc{C})[\mc{S}^{-1}]$ such that $(\mbf{s},c,\lambda,\Phi)$ is an adjoint relative equivalence 
between the relative categories  $(\simp\mc{C},\mc{S})$ and $(\mc{C},\mc{W})$. The argument is just a slight variant of the one in \cite[Theorem 5.1]{R}.\\ 
We claim that (S2), (S3)' and (S5) imply that given a simplicial homotopy equivalence $e:X\rightarrow Y$ in 
$\simp\mc{C}$ then $\mbf{s}(e)\in\mc{W}$. To see this, it suf{f}{i}ces to check that the image under the simple functor
of the simplicial morphism $d_0^X : X\rightarrow X \otimes \Dl[1]$ is a weak equivalence for each simplicial object $X$. Note that $X\otimes \Dl[1]$ is the diagonal
of the bisimplicial object $X\boxtimes \Dl[1]$ given in bidegree $n,m$ by $(X\boxtimes \Dl[1])_{n,m} = \amalg_{\Dl_m [1]} X_n$. Also $d_0^X$ is the diagonal of 
the bisimplicial map $F:X\times\Dl\rightarrow X\boxtimes \Dl[1]$ such that $F_{n,m} : X_n\rightarrow \amalg_{\Dl_m[1]} X_n$ is induced by $(d_0)_m:\Dl_m[0] = \ast 
\rightarrow \Dl_m[1]$. For a f{i}xed $n\geq 0$, $F_{n,\cdot} $ agrees with $d_0^{X_n}:X_n\rightarrow X_n\otimes \Dl[1]$, so $\mbf{s}(m\rightarrow F_{n,m})$ is in 
$\mc{W}$ by (S5). By (S2) we deduce that 
$\mbf{s}(n\rightarrow \mbf{s}(m\rightarrow F_{n,m}))$ is in $\mc{W}$ as well, so by (S3)' we conclude that $\mbf{s}(n\rightarrow F_{n,n}) = \mbf{s}(d_0^X)$
is in $\mc{W}$ as required.
Therefore the claim holds, and in particular the simple of an augmentation with an extra degeneracy is a weak equivalence.\\
Given an object $X$ in $\simp\mc{C}$, consider its associated Illusie's bisimplicial decalage $dec(X)$ with $dec(X)_{n,m}=X_{n+m+1}$. 
It comes equipped with the augmentations $\Lambda^I: dec(X) \rightarrow \Dl\times X$ and $\Lambda^{II}: dec(X) \rightarrow X\times \Dl$. They are such that 
$\Lambda^{II}$ has an extra degeneracy by rows and  $\Lambda^I$ has an extra degeneracy by columns. Therefore, by previous claim and (S3)' we deduce that applying $\mbf{s}$ with 
respect to the second index produces the zigzag 
$$ \{\mbf{s}c X_k\}_k \longleftarrow \mbf{s}(m\rightarrow dec(X)_{\cdot,m}) \longrightarrow c\mbf{s}(X)$$   
that is pointwise in $\mc{S}$. By def{i}nition of these augmentations, this zigzag is the identity on $c\mbf{s}c(A)$ in case $X$ is the constant simplicial object equal to $A$ in each degree.
We def{i}ne $\Phi : 1_{\simp\mc{C}}\dashrightarrow c\mbf{s}$ on $X$ as the composition of the above zigzag with the inverse of $\lambda_{X_k} : \{\mbf{s}c X_k\}_k\dashrightarrow \{X_k\}_k$. 
This is possible because the class $\mc{W}$ of pointwise weak equivalences of $\simp\mc{C}$ is contained in $\mc{S}$ by (S2). 
It follows that $\{\lambda_{X_k}\}_k$ is an isomorphism in $Fun(\simp\mc{C},\simp\mc{C})[\mc{S}^{-1}]$. Therefore $\Phi:1_{\simp\mc{C}}\dashrightarrow c\mbf{s}$ is an isomorphism of $Fun(\simp\mc{C},\simp\mc{C})[\mc{S}^{-1}]$ such that the composition 
$c\lambda \, \Phi_c : c\dashrightarrow c$ is the identity in $Fun(\mc{C},\simp\mc{C})[\mc{S}^{-1}]$. Finally, note that this formally implies the other triangle 
identity: $\lambda_{\mbf{s}}\, \mbf{s}\Phi:\mbf{s}\dashrightarrow \mbf{s}$ is the identity in $Fun(\simp\mc{C},\mc{C})[\mc{S}^{-1}]$.  
Indeed, since $\lambda$ is an isomorphism and a natural transformation then $\mbf{s}c  \lambda_{\mbf{s}} = \lambda_{\mbf{s}c\mbf{s}}$ because 
$\lambda_{\mbf{s}}\, \mbf{s}c\lambda_{\mbf{s}} =  \lambda_{\mbf{s}} \lambda_{\mbf{s}c\mbf{s}}$. Then $c\lambda \, \Phi_c = 1_{c}$ implies 
$\mbf{s}c\lambda_{\mbf{s}} \, \mbf{s}\Phi_{c\mbf{s}} = 1_{\mbf{s}c\mbf{s}}$, so $\lambda_{\mbf{s}c\mbf{s}} \, \mbf{s}\Phi_{c\mbf{s}} = 1_{\mbf{s}c\mbf{s}}$. In 
view of the commutative diagram
$$\xymatrix@M=4pt@H=4pt@C=25pt{ \mbf{s} \ar[r]^{\mbf{s}\Phi} \ar[d]_{\mbf{s}\Phi} & \mbf{s}c\mbf{s}  \ar[r]^{\lambda_{\mbf{s}}}  \ar[d]^{\mbf{s}c\mbf{s}\Phi} & \mbf{s}  \ar[d]^{\mbf{s}\Phi}\\
\mbf{s}c\mbf{s} \ar[r]^{\mbf{s}\Phi_{c\mbf{s}}}  & \mbf{s}c\mbf{s}c\mbf{s}  \ar[r]^{\lambda_{\mbf{s}c\mbf{s}}} & \mbf{s} c\mbf{s}
}$$
we deduce $\lambda_{\mbf{s}}\, \mbf{s}\Phi=1_{\mbf{s}}$ as claimed.\\
We are now ready to construct $\mu:\mbf{s}\mrm{D}\dashrightarrow \mbf{s}\mbf{s}$ compatible with $\lambda$. Given a bisimplicial object $Z$, performing the previous 
construction of $\Phi$ to $Z$ with respect to the second index gives $(\simp\Phi)_Z = \{(\Phi_{Z_{n,\cdot}})_m\}_{n,m} : Z\dashrightarrow (\simp (c\mbf{s}))(Z) = \{\mbf{s}(k\rightarrow Z_{n,k})\}_{n,m}$.   
This produces the isomorphism $\simp\Phi : 1_{\simp\simp\mc{C}}\dashrightarrow \simp (c\mbf{s})$ in $Fun(\simp\simp\mc{C},\simp\simp\mc{C})[\simp\mc{S}^{-1}]$, where
$\simp\mc{S}$ is the class of $\simp\simp\mc{C}$ formed by those morphisms $F$ such that $\mbf{s}(k\rightarrow F_{n,k})\in\mc{S}$ for each $n\geq 0$. 
By (S3)' we have that $\mrm{D}(\simp\mc{S})\subset \mc{S}$. It follows that applying $\mrm{D}$ to $\simp\Phi$ we obtain the isomorphism
$\mrm{D}(\simp\Phi) : \mrm{D}\dashrightarrow \simp\mbf{s}$ of $Fun(\simp\simp\mc{C},\simp\mc{C})[\mc{S}^{-1}]$, so applying $\mbf{s}$ we get the desired isomorphism
$\mu = \mbf{s}\mrm{D}(\simp\Phi) : \mbf{s} \mrm{D}\dashrightarrow \mbf{s}\mbf{s}$ of $Fun(\simp\simp\mc{C},\mc{C})[\mc{W}^{-1}]$. Finally, given a simplicial object $X$
the compositions (\ref{compatibLambdaMuEquac}) are obtained by triangle identities satisf{i}ed by $\Phi$ and $\lambda$, and therefore they are equal to the identity.
\end{proof}

\begin{obs} Previous proposition comes to say that in the notion of simplicial descent category the axiom (S3) may be replaced by the weaker axiom 
(S3)'. Also (S1) could be weakened, indeed the preservation of f{i}nite coproducts by the simple functor 
$\mbf{s}$ can be deduced from the remaining axioms. We will not use this fact in the present paper, however.
\end{obs}
\vspace{0.1cm}
\subsection{Simplicial descent structures provide realizable homotopy colimits.}\mbox{}\\[0.2cm]
\indent
In this subsection we prove that a simplicial descent structure produces, in the case of exact coproducts, realizable homotopy colimits
given by a formula following the pattern of Bousf{i}eld-Kan and Voevodsky. We will then work with a simplicial descent category $(\mc{C},\mc{W},\mbf{s})$ 
such that $(\mc{C},\mc{W})$ is closed by coproducts. In this case we will just say that $(\mc{C},\mc{W},\mbf{s})$ is a simplicial descent category closed 
by coproducts.

\begin{defi}\label{HocolimDesc} 
Let $(\mc{C},\mc{W},\mbf{s})$ be a simplicial descent category closed by coproducts.
Given a small category $I$, def{i}ne $\mtt{hocolim}_I : \mc{C}^I \rightarrow \mc{C}$ as the composition
$$\xymatrix@M=4pt@H=4pt@C=35pt{
 \mc{C}^I \ar[r]^-{\coprod^I} \ar[rd]_-{\mtt{hocolim}_I\, }&  \simp\mc{C} \ar[d]^-{\mbf{s}}\\
  & \mc{C}}$$
of the simple functor with the simplicial replacement.
If needed, we will write $\mtt{hocolim}_I^\mc{C}$ to emphasize the target category where we are taking homotopy colimits.
\end{defi}

\begin{obs}
Note that if $X:I\rightarrow\mc{C}$ is a diagram and $c:\mc{C}\rightarrow\simp\mc{C}$ is the constant functor then by def{i}nition
\begin{equation}\label{hocolims} \mtt{hocolim}_I X = \mbf{s} ( \mtt{hocolim}_I^{\mrm{V}} cX )\end{equation}
where $\mtt{hocolim}_I^{\mrm{V}}: \simp\mc{C}^I\rightarrow\simp\mc{C}$ is the homotopy colimit given in def{i}nition \ref{def{i}HolimDl}.
\end{obs}

In the next proposition we prove that $\mtt{hocolim}_I$ is a relative functor invariant under homotopy right cof{i}nal changes of diagrams.

\begin{prop}\label{HomotopyInvariance}
Let $(\mc{C},\mc{W})$ be simplicial descent category closed by coproducts. Given diagrams $X,Y:I\rightarrow \mc{C}$ and
a natural transformation $\rho : X \rightarrow Y$ with $\rho_i\in\mc{W}$ for all $i\in I$, then $\mtt{hocolim}_I \rho$ is also in $\mc{W}$.
\end{prop}

\begin{proof}
Since $\mc{W}$ is assumed to be closed by coproducts, we have that $\amalg^I \rho : \amalg^I X \rightarrow \amalg^I Y$ is a simplicial map which is degree-wise in
$\mc{W}$. Therefore, axiom (S2) ensures that $\mtt{hocolim}_I \rho = \mbf{s}(\amalg^I\rho) \in \mc{W}$.
\end{proof}

By construction $\mtt{hocolim}_I$ is natural on $I$. Recall that given two diagrams $X:I\rightarrow \mc{C}$ and $Y:J\rightarrow \mc{C}$, a morphism
$(f,\tau):X\rightarrow Y$ between them is a functor $f:X\rightarrow Y$ plus a natural transformation $\alpha : f^\ast X\rightarrow Y$. Then we have the natural 
morphism of $\mc{C}$ 
$$\mtt{hocolim} (f,\tau) : \mtt{hocolim}_I X\rightarrow \mtt{hocolim}_J Y $$
def{i}ned as the simple of the simplicial morphism $\amalg^{\bullet}(f,\tau):\amalg^I X\rightarrow \amalg^J Y$ given in (\ref{SRnatural}).

\begin{prop}\label{cof{i}nalitythm}
Let $(\mc{C},\mc{W})$ be simplicial descent category closed by coproducts and $f:I\rightarrow J$ a
homotopy right cof{i}nal functor. Then, for each diagram $X:J\rightarrow\mc{C}$ the natural map
$$\mtt{hocolim}_I f^\ast X\longrightarrow \mtt{hocolim}_J X$$
induced by the morphism of diagrams $(f,1): f^\ast X\rightarrow X$ is a weak equivalence.
\end{prop}

\begin{proof}
By Proposition \ref{DlClosedDescenso} we have that $\mc{S}=\mbf{s}^{-1}\mc{W}$ is a $\Dl$-closed class of $\simp\mc{C}$, which is closed by coproducts if $\mc{W}$ 
is. If follows from Proposition \ref{hlcDl} that $\mtt{hocolim}_I^{\mrm{V}} f^{\ast}X = \amalg^I f^\ast X \rightarrow \mtt{hocolim}_J^{\mrm{V}} X = \amalg^J X$ is 
in $\mc{S}$. But this is the same as saying that $\mbf{s}(\amalg^I f^\ast X) = \mtt{hocolim}_I f^\ast X
\rightarrow \mbf{s}(\amalg^J X)=\mtt{hocolim}_J X$ is in $\mc{W}$.
\end{proof}

\begin{defi}
Given a functor $f:I\rightarrow J$, the \textit{homotopy left Kan extension} of $f$ is $f_{!} : \mc{C}^I\rightarrow \mc{C}^J$ given by
\begin{equation}\label{lKextpointwise} (f_{!} X) (j) = \mtt{hocolim}_{(f/j)} u_j^\ast X\end{equation}
where $u_j : (f/j)\rightarrow I$ maps $\{f(i)\rightarrow j\}$ to $i$.
\end{defi}

\begin{thm} \label{AdjointPair}
Let $(\mc{C},\mc{W})$ be simplicial descent category closed by coproducts. Under the previous notations, the following
properties hold.\\
\textbf{i.} For any small category $I$, $\mtt{hocolim}_I$ is a realizable homotopy colimit on $(\mc{C},\mc{W})$, invariant under homotopy right 
cof{i}nal changes of diagrams.\\
\textbf{ii.} Given $f:I\rightarrow J$ between small diagrams, $(f_{!}, f^\ast )$ is a relative adjunction between $(\mc{C}^I,\mc{W})$
and $(\mc{C}^J,\mc{W})$. 
\end{thm}

\begin{proof}[\textbf{Proof of Theorem \ref{AdjointPair}.}] 
Since $\mtt{hocolim}_I= \pi_{!}$ where $\pi:I\rightarrow [0]$ is the trivial functor, to see that $\mtt{hocolim}_I$ is a realizable homotopy colimit it suf{f}{i}ces to 
prove the parts \textit{ii}.
Fix a simple functor $\mbf{s}:\simp\mc{C}\rightarrow\mc{C}$. By  Proposition \ref{DlClosedDescenso} $\mc{S}=\mbf{s}^{-1}\mc{W}$ is a $\Dl$-closed class of $\simp\mc{C}$ which is closed by coproducts if $\mc{W}$ is.
Given $f:I\rightarrow J$, consider the following diagram of functors of relative categories
$$\xymatrix@M=4pt@H=4pt@C=33pt{ (\mc{C}^I,\mc{W}) \ar@<0.5ex>[r]^-{c}  & (\simp\mc{C}^I,\mc{S}) \ar@<0.5ex>[l]^-{\mbf{s}^I}
\ar@<0.5ex>[r]^-{\,\, f_{!}^{\simp\mc{C}}}  & (\simp\mc{C}^J,\mc{S}) \ar@<0.5ex>[l]^-{f^\ast}
\ar@<0.5ex>[r]^-{\mbf{s}^J} & (\mc{C}^J,\mc{W}) \ar@<0.5ex>[l]^-{c}  }$$
By Theorem \ref{hocolimDlClosed} $(f_{!}^{\simp\mc{C}},f^\ast)$ is a relative adjunction.
By Proposition \ref{DescensoFuntores} $(\mc{C}^I,\mc{W})$ and $(\mc{C}^J,\mc{W})$ are simplicial descent categories with simple functor def{i}ned pointwise. Note 
that $(\mbf{s}^I)^{-1}\mc{W}$ agrees with the class of $\simp\mc{C}^I$ def{i}ned pointwise by $\mc{S}$, and analogously for $J$. We conclude by Theorem 
\ref{equivCat} that the pairs $(c,\mbf{s}^I)$ and $(\mbf{s}^J,c)$ are relative adjoint equivalences of categories, and in particular relative adjunctions.
It turns out that $\mbf{s}^J\, f_{!}^{\simp\mc{C}} \, c = f_{!}^{\mc{C}}$ and
$\mbf{s}^I\, f^\ast \, c $ form a relative adjunction as well. But $\mbf{s}^I f^\ast  c (A) = \mbf{s}^I c f^\ast (A)$ is naturally equivalent to $f^\ast (A)$ 
by (S4). Hence, $(f_{!}^{\mc{C}},f^\ast)$ is a relative adjunction.
To f{i}nish, the invariance of $\mtt{hocolim}_I =  \mbf{s}\, \mtt{hocolim}_I^{\mrm{V}}\, c$ under homotopy right cof{i}nal changes of diagrams 
follows from the invariance of $\mtt{hocolim}_I^{\mrm{V}}$ under them.
\end{proof}
\vspace{0.1cm}
\subsection{Realizable homotopy colimits provide simplicial descent structures.}\mbox{}\\[0.2cm]
\indent
We f{i}nish here the proof of Theorem \ref{characterization}, with the following proposition.  

\begin{prop}
Let $(\mc{C},\mc{W})$ be a relative category closed by coproducts. Assume that $(\mc{C},\mc{W})$ admits realizable homotopy colimits 
$\mtt{hocolim}_I:\mc{C}^I\rightarrow \mc{C}$, which are invariant under homotopy right cof{i}nal changes of diagrams.
Then $(\mc{C},\mc{W})$ admits a simplicial descent structure with simple functor $\mbf{s}=\mtt{hocolim}_{\simp}:\simp\mc{C}\rightarrow \mc{C}$, and 
there is a unique isomorphism of $\mc{R}el\mc{C}at$ between $\mtt{hocolim}_{I}$ and $\mbf{s}\amalg^I$ compatible with the adjunction morphisms.
\end{prop}

\begin{proof} Under the above assumptions, let us see that $\mbf{s}:=\mtt{hocolim}_{\simp}:\simp\mc{C}\rightarrow\mc{C}$ endows $(\mc{C},\mc{W})$ with a 
simplicial descent structure.\\[0.05cm]
\indent (S1): By Lemma \ref{reladjLoc}
${\mbf{s}}:\simp\mc{C}[\mc{W}^{-1}]\rightarrow\mc{C}[\mc{W}^{-1}]$ is a left adjoint. Then, arguing as in Lemma \ref{SimpleCoprod} we deduce that 
$\mbf{s}:\simp\mc{C}\rightarrow\mc{C}$ commutes with f{i}nite coproducts up to weak equivalence.\\[0.05cm]
\indent (S2): By hypothesis, $\mbf{s}:(\simp\mc{C},\mc{W})\rightarrow (\mc{C},\mc{W})$ is a relative functor. This means 
that it sends pointwise weak equivalences to weak equivalences.\\[0.05cm]
\indent (S3): The diagonal $d:\simp\rightarrow\simp\times\simp$ is a homotopy right cof{i}nal functor (see \cite[Lemma 5.33]{T}). By the cof{i}nality 
property of our realizable homotopy colimits we deduce an induced isomorphism 
$\mtt{hocolim}(d):\mbf{s}\mrm{D}\dashrightarrow \mtt{hocolim}_{\simp\times\simp}$ of $Fun(\simp\simp\mc{C},\mc{C})[\mc{W}^{-1}]$. 
By Lemma \ref{reladjDiag} $\mtt{hocolim}_{\simp}^{\simp}: \simp\simp\mc{C}\rightarrow \simp\mc{C}$ is relative left adjoint to $c^{\simp}:\simp\mc{C}\rightarrow \simp\simp\mc{C}$.
Then $\mtt{hocolim}_{\simp}\,\mtt{hocolim}_{\simp}^{\simp}=\mbf{s}\mbf{s}$ is another relative left adjoint of 
$c_{\simp\times\simp}:\mc{C}\rightarrow \simp\simp\mc{C}$. We conclude by 
Proposition \ref{UniqRelAd} that there is a unique isomorphism $k:\mtt{hocolim}_{\simp\times\simp}\dashrightarrow\mbf{s}\mbf{s}$ of 
$Fun(\simp\simp\mc{C},\mc{C})[\mc{W}^{-1}]$ compatible with 
the adjunction morphisms. We def{i}ne $\mu$ as the composition $\mtt{hocolim}(d)$ and $k$.\\[0.05cm]
\indent (S4): Def{i}ne $\lambda:\mbf{s}\, c\dashrightarrow 1_{\mc{C}}$ as the adjunction morphism of $(\mbf{s},c)$. We claim that
$\lambda$ is an isomorphism of $Fun(\mc{C},\mc{C})[\mc{W}^{-1}]$. On the one hand, the trivial functor $\pi:\simp\rightarrow [0]$ is homotopy right cof{i}nal 
because $(\pi/0)\equiv \simp$ has an initial object. Then the induced morphism $ \mtt{hocolim}(\pi) : \mtt{hocolim}_{\simp} \pi^\ast  = \mbf{s} c \dashrightarrow 
\mtt{hocolim}_{[0]}$ is an isomorphism. On the other hand, the adjunction morphism 
$u:\mtt{hocolim}_{[0]}\dashrightarrow 1_{\mc{C}}$ of $(\mtt{hocolim}_{[0]},1_{\mc{C}})$ is an isomorphism as well. Finally, it holds that $\lambda = u \, \mtt{hocolim}(\pi)$, so the claim is proved.\\
The compatibility between $\mu$ and $\lambda$ follows from the Fubini property of $\mtt{hocolim}$ (see Proposition \ref{Fubini}) together with the fact that $d:\simp\rightarrow \simp\times\simp$ composed with the projections $p_1,p_2:\simp\times\simp\rightarrow\simp$
is the identity.\\[0.05cm]
\indent (S5): Note that $-\otimes\Dl[1] : \mc{C}\rightarrow \simp\mc{C}$ is left adjoint to the evaluation at $[1]$, $ev_1 : \simp\mc{C}\rightarrow\mc{C}$, 
$ev_1X=X_1$. In addition, the adjunction morphism $ev_1(-)\otimes \Dl[1]\rightarrow 1_{\simp\mc{C}}$ is, in the degree $n$, 
$X_1\otimes\Dl[1]=\amalg_{\alpha:[n]\rightarrow [1]} X_1\rightarrow X_n$ given by $\{X(\alpha):X_1\rightarrow X_n\}_{\alpha}$. In case $X$ is constant equal to 
$A$, this morphism is then equal to $A\otimes s^0: A\otimes \Dl[1]\rightarrow A\otimes \Dl[0]$.\\
Since both $-\otimes\Dl[1]$ and $ev_1$ preserve weak equivalences, it turns out that they still give an adjoint pair after localizing by $\mc{W}$.
It follows from the fact that the composition of adjunctions is an adjunction that
$\mbf{s} (-\otimes \Dl[1]) : \mc{C}[\mc{W}^{-1}]\rightarrow\mc{C}[\mc{W}^{-1}]$ is left adjoint to $ev_1\, c = 1_{\mc{C}[\mc{W}^{-1}]}$. In particular 
$\mbf{s} (-\otimes \Dl[1])$ is isomorphic to $1_{\mc{C}[\mc{W}^{-1}]}$, through the adjunction map $\mbf{s} (-\otimes \Dl[1])\rightarrow 1_{\mc{C}[\mc{W}^{-1}]}$.
Given an object $A$ of $\mc{C}$, this adjunction map is the composition of the isomorphism $\lambda_A$ of (S4) with 
$\mbf{s}(A\otimes s^0)\mbf{s}(A\otimes\Dl[1])\rightarrow \mbf{s}c (A)$. Then $\mbf{s}(A\otimes s^0)$ is a weak equivalence because it is an isomorphism 
of $\mc{C}[\mc{W}^{-1}]$. But then $\mbf{s}(A\otimes d^0)$ is also a weak equivalence because $s^0\, d^0 = 1_{\Dl[0]}$.
\end{proof}

We close the section with the study of the preservation of homotopy colimits by relative functors.

\begin{defi}
Consider relative categories $(\mc{C},\mc{W})$ and $(\mc{D},\mc{W})$, and realizable homotopy colimits 
$$\mtt{hocolim}_I^{\mc{C}} :  (\mc{C}^I,\mc{W}) \rightleftarrows  (\mc{C},\mc{W}) : c_I \ \ \ \ \ \ \mtt{hocolim}_I^{\mc{D}} :  (\mc{D}^I,\mc{W}) \rightleftarrows  (\mc{D},\mc{W}) : c_I$$ 
Given a relative functor $F:(\mc{C},\mc{W})\rightarrow(\mc{D},\mc{W})$ we have a natural relative transformation
$$\rho_F^I: \mtt{hocolim}^{\mc{D}}_I F \dashrightarrow F\,\mtt{hocolim}^{\mc{C}}_I$$
of $Fun(\mc{C}^I,\mc{D})[\mc{W}^{-1}]$ induced by adjunction from 
$$F(\beta_I^{\mc{C}}): F\dashrightarrow  F c_I\mtt{hocolim}_I^{\mc{C}} = c_I F\,\mtt{hocolim}_I^{\mc{C}}$$ 
Here $\beta_I^{\mc{C}}:1_{\mc{C}^{I}}\dashrightarrow c_I\mtt{hocolim}_I^{\mc{C}}$ is the adjunction morphism of $(\mtt{hocolim}_I^{\mc{C}},c_I)$.\\ 
We say that $F$ \textit{commutes with $I$-homotopy colimits} if $\rho_F^I$ is an isomorphism. Note that this def{i}nition does not depend on the representatives 
chosen for $\mtt{hocolim}_I^{\mc{C}}$ and $\mtt{hocolim}_I^{\mc{D}}$.\\
If $F$ commutes with $I$-homotopy colimits for each small category $I$ we simply say that $F$ \textit{commutes with homotopy colimits}.
In case $F$ commutes with $\Omega$-homotopy colimits for each discrete category $\Omega$, we say that $F$ \textit{commutes with homotopy coproducts}.
\end{defi}

\begin{cor}\label{morphHCC} Let $F:(\mc{C},\mc{W})\rightarrow(\mc{D},\mc{W})$ be relative functor between relative categories closed by coproducts and 
admitting realizable homotopy colimits which are invariant under homotopy right cof{i}nal changes of diagrams. Then $F$ commutes with homotopy colimits 
if and only if it commutes with $\simp$-homotopy colimits and with homotopy coproducts. 
\end{cor}

\begin{proof} Consider a relative functor $F:(\mc{C},\mc{W})\rightarrow(\mc{D},\mc{W})$ and a diagram $X:I\rightarrow \mc{C}$. Since $F$ commutes with homotopy coproducts then for each $n\geq 0$ the canonical morphism 
$$\amalg_{i_0\rightarrow\cdots\rightarrow i_n} F(X_{i_0}) \longrightarrow F(\amalg_{i_0\rightarrow\cdots\rightarrow i_n} X_{i_0})$$
is an isomorphism of $\mc{D}[\mc{W}^{-1}]$, and therefore an equivalence. It turns out that the canonical morphism $\varrho:\amalg^I F \rightarrow F\amalg^I$ is a pointwise weak equivalence.
Under the assumptions on $(\mc{C},\mc{W})$ and $(\mc{D},\mc{W})$, we know that $\mtt{hocolim}_{\simp}^{\mc{C}}\amalg^I$
 and $\mtt{hocolim}_{\simp}^{\mc{D}}\amalg^I$ are $I$-homotopy colimits. Then $\rho_F^I$ is an isomorphism since it is the composition of the weak equivalence
$ \mtt{hocolim}_{\simp}^{\mc{C}} (\varrho): \mtt{hocolim}_{\simp}^{\mc{C}}\amalg^I F \rightarrow \mtt{hocolim}_{\simp}^{\mc{C}}F\amalg^I$ with the isomorphism
$\rho_F^{\simp}\compc \amalg^I : \mtt{hocolim}_{\simp}^{\mc{C}}F\amalg^I\dashrightarrow F \mtt{hocolim}_{\simp}^{\mc{C}}\amalg^I$.
\end{proof}

\begin{obs} Previous corollary together with Theorem \ref{characterization} can be stated as an equivalence of categories 
between the category formed by the simplicial descent categories and the category formed by the relative categories 
closed by coproducts and possessing realizable homotopy colimits invariant under homotopy right cof{i}nal changes of diagrams.
For brevity's sake, the details are left to the reader.
\end{obs}


\section{Bousf{i}eld-Kan homotopy colimits are realizable.}\label{sectionBK}

In this section we study the Bousf{i}eld-Kan homotopy colimits, and prove the 

\begin{thm}\label{MCHCbis} Let $(\mathcal{M},\mathcal{W})$ be a model category. Then $(\mc{M},\mc{W})$ admits realizable homotopy colimits, which are invariant 
under homotopy right cof{i}nal changes of diagrams. In addition, 
they may be computed using the corrected Bousf{i}eld-Kan formula
\begin{equation}\label{cBKreal}\mtt{{}_c hocolim}^{BK}_I X = \int^i  \widetilde{QX}(i) \otimes \mrm{N}(i/I)^{\comp}\end{equation}
\end{thm}

By Proposition \ref{DerivedFunctorColim}, the localization of a realizable homotopy colimit gives an absolute left derived functor of the colimit. 
Then, it follows that the global and local notions of homotopy colimit coincide for any model category. 

\begin{cor}
 Let $(\mc{M},\mc{W})$ be a model category. Given a small category $I$, the functor 
$\mtt{hocolim}^{BK}_I :\mc{M}^I[\mc{W}^{-1}]\rightarrow \mc{M}[\mc{W}^{-1}]$ induced by the corrected Bousf{i}eld-Kan formula \emph{(\ref{cBKreal})}
on localizations is the absolute left derived functor $\mbb{L}\mtt{colim}_I$ of the colimit $\mtt{colim}_I : \mc{M}^I\rightarrow\mc{M}$.
\end{cor}

The proof of previous theorem is based on Theorem \ref{characterization}, and will be f{i}nished at the end of the section.
\vspace{0.2cm}
\subsection{Reminder of Bousf{i}eld-Kan homotopy colimits.}\mbox{}\\[0.2cm]
\indent
To deal with Bousf{i}eld-Kan homotopy colimits in general model categories, we need the machinery of frames and homotopy function complexes, which we now recall. 
Here we use the notations and conventions of \cite{H}, where the reader is referred to for further detail.

\begin{defi} Let $(\mc{M},\mc{W})$ be a model category. A \textit{cosimplicial frame} on an object $A$ of $\mc{M}$ is a cosimplicial object $\widetilde{A}:\Delta\rightarrow \mathcal{M}$ plus a
pointwise weak equivalence $\epsilon:\widetilde{A}\rightarrow cA$ such that 
$\epsilon^0:\widetilde{A}^0\rightarrow A$ is an isomorphism. If $A$ is a cof{i}brant object of $\mathcal{M}$ then,
in addition, $\widetilde{A}$ is assumed to be Reedy cof{i}brant in $\ensuremath{\Delta} \mathcal{M}$.\newline
Any model category possesses a functorial cosimplicial frame $\,\widetilde{\mbox{\,}}\, :%
\mathcal{M}\rightarrow\ensuremath{\Delta} \mathcal{M}$, which is homotopically unique. Given a small category $I$, we also denote by $\,\widetilde{\mbox{\,}}\,:\mathcal{M}^I\rightarrow
\ensuremath{\Delta} \mathcal{M}^I$ the functor def{i}ned pointwise as $\widetilde{X}(i)=\widetilde{X(i)}$, for each diagram $X:I\rightarrow \mc{M}$. A model category 
equipped with a f{i}xed (functorial) cosimplicial frame is called a \textit{framed model category}.\\

In combination with frames we will also use the action $\otimes : \ensuremath{\Delta} \mathcal{M}\times\Delta^\comp Set \rightarrow \mathcal{M}$ given by 
\begin{equation}\label{actionModelos} X\otimes K = \mathtt{colim}_{\ensuremath{\Delta} K} \,\pi_K^\ast X\end{equation} 
Here $%
\ensuremath{\Delta} K$ denotes the category of simplices of $K$  and $\pi_K:%
\ensuremath{\Delta} K\rightarrow \ensuremath{\Delta} $ sends an $n$%
-simplex of $K$ to $[n]$.
\end{defi}

\begin{defi} Let $(\mathcal{M},\mathcal{W})$ be a framed model category, and denote by $\mc{M}_c$ its subcategory of cof{i}brant objects. 
The \textit{Bousf{i}eld-Kan homotopy colimit}, $\mathtt{hocolim}^{BK}_{I}: \simp\mc{M}_c\rightarrow \mc{M}_c$, is def{i}ned on $X:I\rightarrow\mathcal{M}_c$ as the coend of the bifunctor 
$\widetilde{X}\otimes \mathrm{N}(\cdot/ I)^{\comp}:I\times I^{op}\rightarrow \mathcal{M}$ given by $(i,i^{\prime})\mapsto \widetilde{X}(i)\otimes \mathrm{N}(i^{\prime}/I)^{\comp}$. That is, 
\[
\mathtt{hocolim}^{BK}_{I} X = \int^i \widetilde{X}(i)\otimes \mathrm{N}(i/I)^{\comp} 
\]
Bousfield-Kan homotopy colimits are natural on $I$. Indeed,
a functor $f:I'\rightarrow I$ induces a natural map 
$$\mathtt{hocolim}^{BK}(f):\mathtt{hocolim}^{BK}_{I'}f^\ast X\rightarrow \mathtt{hocolim}^{BK}_{I}X$$ 
def{i}ned by the maps $Id\otimes \mathrm{N}(f)^{\comp}: \widetilde{X}(f(i'))\otimes \mathrm{N}(j^{\prime}/I')^{\comp}%
\rightarrow \widetilde{X}(f(i'))\otimes \mathrm{N}(f(j^{\prime})/I)^{\comp}$. Recall that, by \cite[19.6.7]{H},
 $\mathtt{hocolim}^{BK}(f)$ is a weak equivalence in case $f$ is homotopy right cofinal.
\end{defi}

\begin{ej}
The Bousf{i}eld-Kan homotopy limit of (f{i}brant) simplicial sets, $\mtt{holim}_{I^{\comp}}^{BK}:(\simp Sets)^{I^{\comp}}\rightarrow \simp Sets$, is 
a particular case of the dual construction. Namely, given $K:I^{\comp}\rightarrow (\simp Sets)_f$ then
$$\mtt{holim}_{I^{\comp}}^{BK} X = \int_i {K}(i)\times \mathrm{N}(I/i) $$

A crucial point for the results given later is that $\mtt{holim}_{\Dl}^{BK}:\Dl(\simp Set)_f\rightarrow (\simp Sets)_f$ is a simple functor endowing $((\simp Set)_f,\mc{W})$ with a cosimplicial descent
 structure. This means that the dual of axioms (S1),$\cdots$,(S5) are satisf{i}ed for $((\simp Set)_f,\mc{W},\mtt{holim}_{\Dl}^{BK})$. For a proof the reader may 
consult \cite[Theorem 3.2]{R}. 
\end{ej}

Bousfield-Kan homotopy colimits have the correct homotopical behavior only on pointwise cof{i}brant diagrams.
For this reason, to define it on general diagrams, it is more convenient to `correct' it by composing it with a cofibrant replacement.

\begin{defi}
The \textit{corrected Bousf{i}eld-Kan homotopy colimit}, $\mathtt{{}_c hocolim}^{BK}_{I}: \simp\mc{M}\rightarrow \mc{M}$, is def{i}ned as the 
composition of $\mathtt{hocolim}^{BK}_{I}$ with the functor $Q^I:\mc{M}^I\rightarrow \mc{M}^I$ 
induced pointwise by a functorial cof{i}brant replacement $Q:\mc{M}\rightarrow \mc{M}$. That is,
$\mathtt{{}_c hocolim}^{BK}_{I} X = \mathtt{hocolim}^{BK}_{I} Q^I X$.
\end{defi}

To finish, we recall that for any model category there exists a homotopy function complex $\mrm{map}(\cdot,\cdot):\mc{M}^{\comp}\times\mc{M}\rightarrow \simp Sets$,
that allows to characterize weak equivalences of $\mc{M}$ in terms of weak equivalences of simplicial sets. More precisely, it holds that
$f:X\rightarrow Y$ is a weak equivalence of $\mc{M}$ if and only if for each f{i}brant object $A$ of $\mc{M}$ the induced morphism 
$f^{\ast}:\mrm{map}(Y,A)\rightarrow \mrm{map}(X,A)$ is a weak equivalence of simplicial sets (see \cite[17.7.7]{H}).
\vspace{0.1cm}
\subsection{Bousf{i}eld-Kan homotopy colimit as a simple functor.}

\begin{prop}\label{ModSimpl} Let $(\mathcal{M},\mathcal{W})$ be a model category. Then $(\mathcal{M}_c,\mathcal{W})$ is a simplicial descent category with 
simple functor $\mathbf{s}=\mathtt{hocolim}^{BK}_{\Delta^\comp} : \Delta^\comp\mathcal{M}_c\rightarrow\mathcal{M}_c$.
\end{prop}

\begin{proof}
To begin with, we observe that $(\mathcal{M}_c,\mathcal{W})$ is a relative
category closed by coproducts. The class $\mathcal{W}$ is known to be saturated in $\mc{M}$, and therefore it is saturated in $\mc{M}_c$ as well. On the other hand,
it is a basic property of model categories that cof{i}brant objects and weak equivalences between them are closed by coproducts.
By Proposition \ref{simplif{i}cacionAxiomas}, it suf{f}{i}ces to see that 
axioms (S1),(S2), (S3)', (S4) and (S5) are satisf{i}ed for $(\mathcal{M}_c,\mathcal{W})$ and $\mbf{s}=\mathtt{hocolim}^{BK}_{\Delta^\comp}$.\\[0.1cm]
\indent (S1) Let $X$ ,$Y:\simp\rightarrow \mc{M}_c$
be pointwise cof{i}brant simplicial
objects. If we endow $\mathcal{M}\times \mc{M}$ with the product model
structure, then the coproduct and the codiagonal form a Quillen pair $\coprod : \mathcal{M}\times \mc{M} \leftrightarrows 
\mathcal{M} : \delta $. By \cite[19.4.5]{H} $\mathtt{%
hocolim}_{\simp}^{BK}$ commutes with the left adjoint of a Quillen pair up to
essentially unique zigzag, which is a zigzag of weak equivalences on
pointwise cof{i}brant objects. Therefore, the canonical morphism $\mathtt{hocolim}^{BK}_{\Delta^\comp}X\coprod\mathtt{hocolim}^{BK}_{\Delta^\comp}Y \rightarrow 
\mathtt{hocolim}^{BK}_{\Delta^\comp} (X\coprod Y)$ is a weak equivalence.\\[0.1cm]
\indent (S2) The Bousf{i}eld-Kan homotopy colimit is homotopy invariant on pointwise cof{i}brant diagrams by \cite[19.4.2]{H}.\\[0.1cm]
\indent (S4) The trivial functor $\pi:\simp\rightarrow [0]$ is homotopy right cof{i}nal, then for each cof{i}brant object $B$ of $\mc{M}$ we have the induced natural 
weak equivalence $\mathtt{hocolim}^{BK}_{\simp}c (B)\rightarrow \mathtt{hocolim}^{BK}_{[0]}B$. But by def{i}nition $\mathtt{hocolim}^{BK}_{[0]}B$ is naturally isomorphic 
to $B$. In this way we obtain the natural weak equivalence $\lambda: \mathtt{hocolim}^{BK}_{\simp}c\rightarrow 1_{\mc{M}_c}$.\\[0.1cm]
For the proof of the remaining axioms, we will strongly use the so called `adjunction property' of $\mathtt{hocolim}^{BK}_I$ (\cite[19.4.4]{H}). 
It states that given $X:I\rightarrow \mc{M}_c$ and a f{i}brant object $A$ of $\mc{M}$, there is a natural zigzag of weak equivalences 
between f{i}brant simplicial sets
$$\mrm{map} (\mathtt{hocolim}^{BK}_{I}X, A) \dashrightarrow \mtt{holim}_{I^{\comp}}^{BK}\mrm{map}(X,A)$$
\indent (S3)' Let $F:Z^1\rightarrow Z^2$ be a morphism of $\simp\simp\mc{M}_c$, and let us see that 
$\mbf{s}\mbf{s}(F)=\mbf{s}(n\rightarrow\mbf{s}(m\rightarrow F_{n,m}))$ is a weak equivalence if and only if  $\mbf{s}\mrm{D}(F)$ is. This follows from 
the adjunction property. Indeed, given a f{i}brant object $A$ of $\mc{M}$, we have for $i=1,2$ a natural zigzag of weak equivalences of f{i}brant simplicial sets
$$\mrm{map}(\mtt{hocolim}_{\simp\times\simp}^{BK} Z^i ,A ) \simeq \mtt{holim}_{\Dl\times\Dl}^{BK} \mrm{map}(Z^i,A) \simeq \mtt{holim}_{\Dl}^{BK}\mtt{holim}_{\Dl}^{BK} \mrm{map}(Z^i,A) $$ 
where the last isomorphism is the Fubini property of the homotopy limit of simplicial sets (see, for instance, \cite{T}). 
Again, there is a natural zigzag of weak equivalences $\mtt{holim}_{\Dl}^{BK} \mrm{map}(Z^i,A)\simeq \mrm{map}(\mtt{hocolim}_{\simp}^{BK}Z_i,A)$ 
of f{i}brant simplicial sets. It 
follows from the homotopy invariance of $\mtt{holim}_{\Dl}^{BK}$ on pointwise f{i}brant diagrams that
$$\mtt{holim}_{\Dl}^{BK}\mtt{holim}_{\Dl}^{BK} \mrm{map}(Z^i,A) \simeq 
\mtt{holim}_{\Dl}^{BK}\mrm{map}(\mtt{hocolim}_{\simp}^{BK}Z^i,A)\simeq$$
$$\simeq \mrm{map}(\mtt{hocolim}_{\simp}^{BK}\mtt{hocolim}_{\simp}^{BK}Z^i,A)$$ 
On the other hand, $\mtt{holim}_{\Dl}^{BK}$ is a simple functor on $((\simp Set)_f,\mc{W})$. In particular, for each 
$K:\Dl\times\Dl\rightarrow (\simp Sets)_f$ there is a natural isomorphism $\mtt{holim}_{\Dl}^{BK}\mrm{D} K\rightarrow \mtt{holim}_{\Dl}^{BK}\mtt{holim}_{\Dl}^{BK} K$ in 
$(\simp Set)_f[\mc{W}^{-1}]$. We deduce that 
$$\mrm{map}(\mtt{hocolim}_{\simp}^{BK}\mtt{hocolim}_{\simp}^{BK}Z^i,A)\simeq 
\mtt{holim}_{\Dl}^{BK}\mtt{holim}_{\Dl}^{BK} \mrm{map}(Z^i,A)\simeq$$
$$\simeq \mtt{holim}_{\Dl}^{BK}\,\mrm{D}\, \mrm{map}(Z^i,A)$$ 
By def{i}nition $\mrm{map}(Z^i,A):\Dl\times\Dl\rightarrow \simp (Sets)_f$ is def{i}ned pointwise, which means that it is in bidegree $(n,m)$ equal to 
$\mrm{map}(Z^i,A)_{n,m}=\mrm{map}(Z_{n,m}^i,A)$.
Therefore $\mrm{D} \,\mrm{map}(Z^i,A) = \mrm{map}(\mrm{D} Z^i,A)$, and 
$$\mrm{map}(\mtt{hocolim}_{\simp}^{BK}\mtt{hocolim}_{\simp}^{BK}Z^i,A)\simeq \mtt{holim}_{\Dl}^{BK} \mrm{map}(\mrm{D} Z^i,A) \simeq \mrm{map}(\mtt{hocolim}_{\simp}^{BK}\mrm{D} Z^i,A)$$
We conclude that $\mtt{hocolim}_{\simp}^{BK}\mtt{hocolim}_{\simp}^{BK}(F)$ induces a weak equivalence on homotopy function complexes if and only if 
$\mtt{hocolim}_{\simp}^{BK}\mrm{D}(F)$ does. Then $\mtt{hocolim}_{\simp}^{BK}\mtt{hocolim}_{\simp}^{BK}(F)$ is a weak equivalence if and only if 
$\mtt{hocolim}_{\simp}^{BK}\mrm{D}(F)$ is.\\[0.1cm]
\indent (S5) To avoid confusion with action (\ref{actionModelos}), we denote here the simplicial action $\mc{M}_c\times \simp Set\rightarrow \simp\mc{M}_c$ of
(\ref{ActionSset}) by $\boxtimes$. So, for a cof{i}brant object $B$ of $\mc{M}$ and a simplicial set $K$, $B\boxtimes K$ is now the simplicial object with 
$(B\boxtimes K)_n = \amalg_{K_n} B$.\\
On the other hand, given simplicial sets $K$ and $L$, denote by $K^L$ the cosimplicial simplicial set given in cosimplicial degree $n$ by $ (K^L)^n =\prod_{L_n} K $. This def{i}nes a functor 
$(-)^L : (\simp Set)^{\comp} \rightarrow \Dl(\simp Set)$. Since $\mtt{holim}_{\Dl}^{BK}:\Dl(\simp Set)_f\rightarrow (\simp Sets)_f$ is a simple functor, the dual 
of (S5) means that the image under $\mtt{holim}_{\Dl}^{BK}$ of $d_0^K : K^{\Dl[1]}\rightarrow K^{\Dl[0]}$ is a weak equivalence for each f{i}brant simplicial set $K$.\\
Consider now a cof{i}brant object $B$ of $\mc{M}$. We must check that
$\mathtt{hocolim}^{BK}_{\simp}$ of the simplicial morphism  $d_0^B:B\boxtimes \Dl[0]=cB\rightarrow B\boxtimes \Dl[1]$ is a weak equivalence. We prove this using the adjunction property. 
We must see that $\mathtt{hocolim}^{BK}_{\simp}(d^0_B)$ induces a weak equivalence on homotopy function complexes. But, given a f{i}brant object $A$ of $\mc{M}$ and 
a simplicial set $L$,
$$\mrm{map} (\mtt{hocolim}^{BK}_{\simp}(B\boxtimes L),A)\simeq 
\mtt{holim}^{BK}_{\Dl}\mrm{map}(B\boxtimes L,A)$$
By def{i}nition $\mrm{map}(B\boxtimes L,A)\simeq \mrm{map}(B,A)^L :\Dl\rightarrow (\simp Set)_f$. Indeed, in degree $n$ we have  
$$\mrm{map}(B\boxtimes L,A)_n=\mrm{map}((B\boxtimes L)_n,A) = \mrm{map} (\coprod_{L_n} B,A) \simeq  \prod_{L_n}\mrm{map}(B,A)= (\mrm{map}(B,A)^L)^n$$
Hence $\mrm{map} (\mtt{hocolim}^{BK}_{\simp}(B\boxtimes L),A)\simeq \mtt{holim}^{BK}_{\Dl} (\mrm{map}(B,A)^L)$.
Finally, $\mathtt{hocolim}^{BK}_{\simp}(d^0_B)\in\mc{W}$ because 
$\mrm{map} (\mathtt{hocolim}^{BK}_{\simp}(d^0_B),A)\simeq \mtt{holim}^{BK}_{\Dl} ({d^0}^{\mrm{map}(B,A)})$ is a weak equivalence of simplicial sets.
\end{proof}

\begin{prop}\label{MCHC} Let $(\mathcal{M},\mathcal{W})$ be a model category. Then $(\mc{M}_c,\mc{W})$ admits
realizable homotopy colimits, which are invariant under homotopy right cof{i}nal changes of diagrams. In addition, they may be computed using the Bousf{i}eld-Kan 
formula 
$$\mtt{hocolim}^{BK}_I X = \int^i  \widetilde{X}(i) \otimes \mrm{N}(i/I)^{\comp}$$
\end{prop}

\begin{proof} We have that $(\mc{M}_c,\mc{W})$ is a simplicial descent category with simple functor $\mbf{s}=\mathtt{hocolim}^{BK}_{\Delta^\comp} : \Delta^\comp
\mathcal{M}_c\rightarrow\mathcal{M}_c$ by Proposition \ref{ModSimpl}. To prove this 
result is more convenient to consider the `opposite' simple functor of $\mbf{s}$. More precisely, denote by $\Upsilon:\simp\mc{M}_c\rightarrow \simp\mc{M}_c$ the functor that reverses the order of the face and degeneracy maps of a simplicial object. Since 
$\Upsilon: (\simp\mc{M}_c,\mc{W})\rightarrow (\simp\mc{M}_c,\mc{W})$ is a relative functor and $\Upsilon^2=1_{\simp\mc{M}_c}$, then it is an equivalence of relative categories. We deduce 
that $\mbf{s}'=\mathtt{hocolim}^{BK}_{\Delta^\comp}\Upsilon$ is also a relative left adjoint to $\Upsilon\, c = c:\mc{M}_c\rightarrow \simp\mc{M}_c$. Therefore 
$\mbf{s}'\simeq \mathtt{hocolim}^{BK}_{\Delta^\comp}$ in $Fun(\simp\mc{M}_c,\mc{M}_c)[\mc{W}^{-1}]$, so $\mbf{s}'$ is another simple functor for 
$(\mc{M}_c,\mc{W})$.\\
Since $(\mathcal{M}_c,\mathcal{W})$ is closed by coproducts, it follows from Theorem \ref{characterization} that it admits
realizable homotopy colimits, which are invariant under homotopy right cof{i}nal changes of diagrams. In addition, given a small category $I$ then
$\mbf{s}' \amalg^I :\mc{M}_c^I\rightarrow \mc{M}_c $ is a realizable homotopy colimit on $(\mc{M}_c,\mc{W})$.
Hence, to prove the proposition it suf{f}{i}ces to see that there is a natural weak equivalence $\mbf{s}' \amalg^I \rightarrow \mathtt{hocolim}^{BK}_{I}$.\\
Denote by ${{\amalg}'}^I:\mc{M}_c^I\rightarrow \simp\mc{M}_c$ the functor mapping $X:I\rightarrow \mc{M}_c$ to $\Upsilon \amalg^I X$. Note that $ {{\amalg}'}^I X$ agrees with the 
simplicial object given in degree $n$ by $\amalg_{\alpha:[n]\rightarrow I^{\comp}}X_{\alpha(n)}$. 
It easy to check that ${{\amalg}'}^I$ carries a pointwise cof{i}brant $I$-diagram to a Reedy cof{i}brant simplicial diagram. Indeed, the $n$-latching object $L_n(\amalg^I X)$ agrees with 
$\amalg_{\alpha\in \mrm{d}\mrm{N}_n(I^{\comp})}X_{\alpha(n)}$, where $\mrm{d}\mrm{N}_n(I^{\comp})$ denotes the degenerate $n$-simplices of $\mrm{N}(I^{\comp})$. Hence the 
$n$-latching map 
$$\amalg_{\alpha\in \mrm{d}\mrm{N}_n(I^{\comp})}X_{\alpha(n)}\longrightarrow \amalg_{\alpha\in \mrm{N}_n(I^{\comp})}X_{\alpha(n)}$$ is a cof{i}bration because each 
$X_{\alpha(n)}$ is cof{i}brant. By the same argument, our chosen functorial cosimplicial frame $\,\widetilde{\mbox{\,}}\,:\mc{M}_c\rightarrow \Dl\mc{M}_c$ induces a functorial Reedy cosimplicial 
frame on each simplicial object of the form ${\amalg}'^I X$. Indeed, by def{i}nition the induced
pointwise cosimplicial frame on diagrams $\,\widetilde{\mbox{\,}}\,:\mc{M}_c^I \rightarrow \Dl\mc{M}_c^I$ is such that $\widetilde{X}(i):=\widetilde{X(i)}$ is a 
Reedy cof{i}brant object of $\Dl\mc{M}$ for each $i\in I$. 
Since the coproduct of cosimplicial frames is again a cosimplicial frame, ${\amalg}'^I \widetilde{X}$ is a cosimplicial frame on ${\amalg}'^I {X}$. To see that it is indeed a 
Reedy cosimplicial frame, we must see that ${\amalg}'^I \widetilde{X}$ is Reedy cof{i}brant (with respect to the simplicial index). But again, its $n$-th latching map is
$$L_n(\amalg^I \widetilde{X})\simeq  \amalg_{\alpha\in \mrm{d}\mrm{N}_n(I^{\comp})}\widetilde{X}(\alpha(n))\longrightarrow \amalg_{\alpha\in\mrm{N}_n(I^{\comp})}\widetilde{X}(\alpha(n))$$
which is a Reedy cof{i}bration because each $\widetilde{X}(\alpha(n))$ is by hypothesis Reedy cof{i}brant in $\Dl\mc{M}$.
Then the Bousf{i}eld-Kan natural transformation $\mtt{hocolim}^{BK}_{\simp}{{\amalg}'}^I \cdot \rightarrow |{{\amalg}'}^I \cdot |$ is a pointwise weak equivalence
by \cite[19.8.1]{H}. Recall that 
$$ |{{\amalg}'}^I X | = \int^n ({{\amalg}'}^I \widetilde{X})_n\otimes \Dl[n] = \int^n ({{\amalg}'}^I \widetilde{X}^n)_n =  \int^n \ds\coprod_{\alpha\in\mrm{N}_n(I^{\comp})}\widetilde{X}^n(\alpha(n)) $$
As in the proof of previous proposition $\boxtimes$ denotes the simplicial action (\ref{ActionSset}), so $\boxtimes : \simp\mc{M}_c\times \simp Set \rightarrow \simp\mc{M}_c$ maps $(B,K)$ to the simplicial object $B\boxtimes K$ given in degree $n$ by
$(B\boxtimes K)_n = \coprod_{K_n} B_n$. For f{i}xed $n,m\geq 0$, it holds that 
$$ \ds\coprod_{\alpha\in\mrm{N}_n(I^{\comp})}\widetilde{X}^m(\alpha(n))  = \int^i \widetilde{X}^m(i)\boxtimes \mrm{N}_n(i/I)^{\comp} $$
Therefore, using the Fubini interchange index Theorem for coends we deduce
$$|{{\amalg}'}^I X | \simeq \int^n \int^i \widetilde{X}^n(i)\boxtimes \mrm{N}_n(i/I)^{\comp} \simeq  \int^i \int^n \widetilde{X}^n(i)\boxtimes \mrm{N}_n(i/I)^{\comp} $$
But, f{i}xed $i,j\in J$, it is straightforward to check that there is a natural isomorphism 
$$\int^n \widetilde{X}^n(i)\boxtimes \mrm{N}_n(j/I)^{\comp} \simeq \mtt{colim}_{\Dl (j/I)^{\comp}} \, \pi_{(j/I)^{\comp}}^{\ast} \widetilde{X}(i) = \widetilde{X}(i)\otimes \mrm{N}(j/I)^{\comp} $$
Hence, putting all together we conclude that
$ |{{\amalg}'}^I X | \simeq \int^i\widetilde{X}(i)\otimes \mrm{N}(i/I)^{\comp} = \mtt{hocolim}^{BK}_I X $.
\end{proof}

\begin{proof}[\textbf{Proof of Theorem \ref{MCHCbis}.}] Denote by $i:\mc{M}_c\rightarrow \mc{M}$ the inclusion. 
By def{i}nition, the image of $Q$ is contained in $\mc{M}_c$, so we may consider $Q':\mc{M}\rightarrow \mc{M}_c$ such that $i\,Q'=Q$. Since $Q$ is a cof{i}brant replacement
functor there is a natural weak equivalence $\epsilon:Q\rightarrow 1_{\mc{M}}$. It gives natural weak equivalences $Q'\, i\rightarrow 1_{\mc{M}_c}$ and 
$i\,Q'\rightarrow 1_{\mc{M}}$. Therefore  
$$i:(\mathcal{M}_c,\mc{W})\rightleftarrows (\mathcal{M},\mc{W}):Q'$$ is a relative adjoint equivalence. In particular $(i,Q')$ and $(Q',i)$ are relative
adjoint pairs of functors. Given a small category $I$, by Proposition \ref{reladjDiag} the pointwise adjoint pair induced by $({Q'},i)$ is also a relative adjunction
$Q':(\mathcal{M}^I,\mc{W})\rightleftarrows (\mathcal{M}_c^I,\mc{W}):i$. Therefore we have the relative adjoint pairs
$$\xymatrix@M=4pt@H=4pt@C=33pt{ (\mc{M}^I,\mc{W}) \ar@<0.5ex>[r]^-{Q'}  & (\mc{M}_c^I,\mc{W}) \ar@<0.5ex>[l]^-{i}
\ar@<0.5ex>[r]^-{\mtt{hocolim}_I^{BK}}  & (\mc{M}_c,\mc{W}) \ar@<0.5ex>[l]^-{c_I}
\ar@<0.5ex>[r]^-{i} & (\mc{M},\mc{W}) \ar@<0.5ex>[l]^-{Q'}  }$$
where the relative adjunction in the middle follows from previous proposition. Then the corrected Bousf{i}eld Kan formula, which agrees with $i\,
\mtt{hocolim}_I^{BK}\,Q'$, is a relative left adjoint to $i\,c_I \,Q' = c_I\, i\, Q'$. Since $i\,Q'$ is isomorphic in $\mc{R}el\mc{C}at$ to the identity, then 
$i\,\mtt{hocolim}_I^{BK}\,Q'$ is a relative left adjoint to $c_I$ which means that it is a realizable homotopy colimit. Since $\mtt{hocolim}_I^{BK}$ is invariant 
under homotopy right cof{i}nal changes of diagrams, the same holds for $i\,\mtt{hocolim}_I^{BK}\,Q'$.
\end{proof}

\begin{obs}\label{OtroHocolim} If $\mtt{hocolim}_{\simp}:\simp\mc{M}\rightarrow\mc{M}$ is any realizable homotopy colimit on a model category 
$(\mathcal{M},\mathcal{W})$, 
for instance $\mtt{hocolim}_{\simp}=\mtt{{}_c hocolim}_{\simp}^{BK}$, then another formula for the realizable homotopy colimit
$\mtt{hocolim}_I :\mc{M}^I\rightarrow \mc{M}$ is
$$\mtt{hocolim}_I X = \mtt{hocolim}_{\simp} \amalg^I (Q X) $$ 
Indeed, as we have seen in previous proof, $\mtt{hocolim}_I^{\mc{M}}$ is given by $\mtt{hocolim}_I^{\mc{M}_c} Q$, which agrees with 
$\mtt{hocolim}_{\simp}^{\mc{M}_c} \amalg^I Q$ by Theorem \ref{characterization}. 
Since there is a natural weak equivalence $Q\rightarrow 1$, we deduce that 
$\mtt{hocolim}_{\simp}^{\mc{M}_c}Q \amalg^I Q = \mtt{hocolim}_{\simp}^{\mc{M}}  \amalg^I Q$ is another
realizable homotopy colimit on $(\mc{M},\mc{W})$ as claimed.
\end{obs}


\section{Realizable homotopy colimits and Grothendieck derivators.}\label{sectionDerivator}

Here we study the connection of realizable homotopy colimits with Grothendieck derivators. For a concise exposition of Grothendieck derivator theory, 
the reader is referred to \cite{Gr}. Recall that a \textit{prederivator} is a strict 2-functor $\mathbb{D}: cat^{\circ}\rightarrow \mc{C}at$, and that a 
a \textit{weak right derivator} is a prederivator $\mbb{D}$ satisfying the following 
four axioms:\\
\begin{compactitem}
\item[\textbf{Der 1}.]  Given $I,J$ in $cat$ then the functor $(i^\ast,j^\ast):\mbb{D}(I\amalg J)\rightarrow \mbb{D}(I)\times\mbb{D}(J)$ induced by the canonical inclusions
$i:I\rightarrow I\amalg J$, $j:J\rightarrow I\amalg J$ is an equivalence of categories.
\item[\textbf{Der 2}.] Given $I$ in $cat$ and $i\in I$, denote also by $i:[0]\rightarrow I$ the functor $0\mapsto i$. Then, a morphism $F$ of $\mbb{D}(I)$ such that  
$i^{\ast} F$ is an isomorphism of $\mbb{D}([0])$ for each $i\in I$ is an isomorphism of $\mbb{D}(I)$.
\item[\textbf{Der 3d}.] If $f:I\rightarrow J$ is a functor in $cat$, then $f^{\ast}:\mbb{D}(J)\rightarrow \mbb{D}(I)$ admits a left adjoint $f_{!}:\mbb{D}(I)\rightarrow \mbb{D}(J)$.
\item[\textbf{Der 4d}.] Given $f:I\rightarrow J$ and $j\in J$, consider the diagram
$$\xymatrix@M=4pt@H=4pt@C=35pt{ 
 (f/j) \ar[r]^{u_j} \ar[d]_{\pi} \ar@{}[rd]|{\stackrel{\alpha}{\mbox{\rotatebox[origin=c]{33}{$\Leftarrow$}}}} & I \ar[d]^{f} \\
 [0] \ar[r]_j & J
}$$
where ${\alpha}_{\tau : f(i) \rightarrow j}: fu_j(\tau)\rightarrow j \pi (\tau)$ is $\tau:f(i)\rightarrow j$. 
It gives rise by adjunction to the morphism $\pi_{!} u_j^{\ast} \rightarrow j^{\ast} f_{!}$ of $\mbb{D}([0])$, which is assumed to be
an isomorphism. 
\end{compactitem}

If $\mbb{D}':cat^{\circ}\rightarrow \mc{C}at$ is another weak right derivator, a morphism of weak right derivators $\mbb{D}\rightarrow \mbb{D}'$ is called \textit{right exact} if it preserves homotopy 
left Kan extensions.\\

To avoid set-theoretical problems we assume that for a
given relative category $(\mc{C},\mc{W})$, $\mc{C}at$ contains the localized diagram
categories $\mc{C}^I[\mc{W}^{-1}]$. We deduce from previous results the

\begin{thm}\label{RightDerivator}  Assume that $(\mc{C},\mc{W})$ is a relative category closed by coproducts that admits realizable homotopy colimits, 
which are invariant under homotopy right cof{i}nal changes of diagrams. Then $\mathbb{D}:cat^{\circ}\rightarrow \mc{C}at$, 
$I\mapsto \mc{C}^I[\mc{W}^{-1}]$ is a weak right derivator.
\end{thm}

\begin{proof} Der 1 is clear from the def{i}nition of $\mbb{D}$. To see Der 2, consider a morphism $F$ of $\mbb{D}(I)=\mc{C}^I[\mc{W}^{-1}]$. 
By Lemma \ref{isosLoc}, $F=\tau F' \tau'$, where $F'$ is a morphism of $\mc{C}^I$ and $\tau$, $\tau'$ are isomorphisms of $\mc{C}^I[\mc{W}^{-1}]$. If for each 
$i\in I$ it holds that $i^{\ast} F = F_i = \tau_i {F'}_i {\tau'}_i$ is an isomorphism, then $F'_i$ is an isomorphism as well. Since $\mc{W}$ is saturated then $F'_i\in\mc{W}$ for each $i$. This means
that $F'$ is a weak equivalence of $\mc{C}^I$, then an isomorphism of $\mc{C}^I[\mc{W}^{-1}]$. Then $F=\tau F' \tau'$ is also an isomorphism of $\mc{C}^I[\mc{W}^{-1}]$.
To see Der 3, consider a functor $f:I\rightarrow J$ between small categories. By Theorem \ref{AdjointPair} \textit{ii} we have that $f_{!}:(\mc{C}^J,\mc{W})\rightarrow (\mc{C}^I,\mc{W})$
is a relative left adjoint to $f^{\ast}:(\mc{C}^I,\mc{W})\rightarrow (\mc{C}^J,\mc{W})$. By Lemma \ref{reladjLoc} the induced functor 
$f_{!}:\mc{C}^I[\mc{W}^{-1}]\rightarrow \mc{C}^J[\mc{W}^{-1}]$ is left adjoint to $f^{\ast}:\mc{C}^J[\mc{W}^{-1}]\rightarrow \mc{C}^I[\mc{W}^{-1}]$ as required.
In addition, $f_{!}$ is def{i}ned pointwise by the formula (\ref{lKextpointwise}), so Der 4 holds as well.
\end{proof}

\begin{lema}\label{isosLoc}
Let $(\mc{C},\mc{W})$ be a simplicial descent category. If $I$ is a small category and $F$ is a morphism of $\mc{C}^I[\mc{W}^{-1}]$, there exist isomorphisms 
$\tau$ and ${\tau}'$ of $\mc{C}^I[\mc{W}^{-1}]$ and a morphism $F'$ of $\mc{C}^I$ such that $F=\tau F' {\tau}'$.
\end{lema}

\begin{proof} Given a small category $I$, Proposition \ref{DescensoFuntores} implies that $(\mc{C}^I,\mc{W})$ is again a simplicial descent category. Therefore we may assume that $I=[0]$, and we have to prove the statement for a morphism $F$ of $\mc{C}[\mc{W}^{-1}]$.
Consider a simple functor $\mbf{s}:\simp\mc{C}\rightarrow\mc{C}$ on $(\mc{C},\mc{W})$. By Propositions \ref{DlClosedDescenso} \textit{ii} and \ref{DlClBrown}, the relative category $(\simp\mc{C},\mc{S}=\mbf{s}^{-1}\mc{W})$ is a Brown category of 
cof{i}brant objects. In this case we know that
each morphism $T$ of $\simp\mc{C}[\mc{S}^{-1}]$ is represented by a length-two zigzag (see \cite[Theorem 1]{Br}). That is, 
$T=w^{-1}T'$ where $T'$ and $w$ are morphisms of $\simp\mc{C}$ and $w\in\mc{S}$.\\
Consider now a morphism $F:A\dashrightarrow B$ of $\mc{C}[\mc{W}^{-1}]$. Then the constant simplicial morphism  
$c(F)$ in $\simp\mc{C}[\mc{S}^{-1}]$ is $c(F)=w^{-1}T'$ with $T'$ and $w$ as before.
Then $\mbf{s}c(F)=(\mbf{s}(w))^{-1}\mbf{s}(T')$. By (S4) there is an isomorphism
$\lambda:\mbf{s}c\dashrightarrow 1_{\mc{C}}$ of $Fun(\mc{C},\mc{C})[\mc{W}^{-1}]$. It follows that $\lambda_B \,\mbf{s}c(F) = F\, \lambda_A$ in $\mc{C}[\mc{W}^{-1}]$. Then
 $F=\lambda_B(\mbf{s}(w))^{-1}\mbf{s}(T') \lambda_A^{-1}$ and the statement 
holds for $\tau = \lambda_B(\mbf{s}( w))^{-1}$, $F'=\mbf{s}(T')$ and ${\tau}'=\lambda_A^{-1}$.
\end{proof}
 
\begin{obs} One can consider the notion of `realizable' right Grothendieck derivator, consisting of a strict 2-functor $\mbb{D}:cat^{\comp}\rightarrow
\mc{R}el\mc{C}at$ (instead of $cat^{\comp}\rightarrow\mc{C}at$) satisfying the corresponding axioms Der 1',$\cdots$, Der 4'.  
In this context previous theorem would mean that, in the exact coproducts case, the existence of realizable homotopy colimits on $(\mc{C},\mc{W})$, which are 
invariant under homotopy right cof{i}nal changes of diagrams, guarantees that $\mbb{D}(I)=(\mc{C}^I,\mc{W})$ is a realizable right Grothendieck derivator.
This variant of Grothendieck's derivator theory will be developed in an independent work (in preparation), in which we give a 2-categorical reformulation of 
the Dwyer-Hirschhorn-Kan-Smith treatment of homotopy colimits (\cite{DHKS}).
\end{obs}

\begin{obs} To possess realizable homotopy colimits
allows one to def{i}ne cof{i}ber sequences on $\mc{C}[\mc{W}^{-1}]$ in a natural way, in case $\mc{C}$ is a pointed category. Indeed, we can consider the \textit{cone}
functor $\mc{C}one(f):Fl(\mc{C})\rightarrow \mc{C}$ def{i}ned as 
$$ \mc{C}one(f) = \mtt{hocolim} \{\ast \leftarrow X\stackrel{f}{\rightarrow} Y\} $$ 
and the \textit{suspension}  $\Sigma :\mc{C}\rightarrow\mc{C}$ def{i}ned as $\Sigma X= {C}one (X\rightarrow \ast)$. 
Under the hypothesis of previous theorem, it can be proved that the cof{i}ber sequences produced through these cone and suspension functors endow
$\mc{C}[\mc{W}^{-1}]$ with a left triangulated structure, and that 
$\mc{C}[\mc{W}^{-1}]$ is a Verdier triangulated category in case 
$\Sigma:\mc{C}[\mc{W}^{-1}]\rightarrow \mc{C}[\mc{W}^{-1}]$ is an equivalence of categories.
\end{obs}

\noindent We deduce the well-known fact that a model category produces a Grothendieck derivator (c.f. \cite{C}).

\begin{cor}
Let $(\mathcal{M},\mathcal{W})$ be a model category. Then the prederivator $\mbb{D}:cat^{\comp}\rightarrow \mc{C}at$ def{i}ned as 
$\mbb{D}(I)=\mc{M}^I[W^{-1}]$ is a Grothendieck derivator.
\end{cor}

\begin{proof} We have that $(\mc{M}_c,\mc{W})$ is a relative category closed by coproducts and a simplicial descent category by Proposition \ref{MCHC}. 
Then by Theorem \ref{RightDerivator} the prederivator associated with $(\mc{M}_c,\mc{W})$ is 
a weak right derivator. A cof{i}brant replacement functor $Q:\mc{M}\rightarrow \mc{M}_c$ gives a natural equivalence of categories 
$\mc{M}^I[\mc{W}^{-1}]\simeq \mc{M}_c^I[\mc{W}^{-1}]$, and then an isomorphism of prederivators. Hence, the prederivator associated with $(\mc{M},\mc{W})$ is also 
a weak right derivator. But it is also a weak left derivator by the dual fact, therefore it is a Grothendieck derivator.
\end{proof}

In the context of colimits, it holds that a functor commuting with coproducts and coequalizers commutes with all homotopy colimits and left Kan extensions.
A homotopical version of this fact is the following

\begin{prop}
 Let $F:(\mc{C},\mc{W})\rightarrow(\mc{D},\mc{W})$ be a relative functor between relative categories closed by coproducts and 
admitting realizable homotopy colimits which are invariant under homotopy right cof{i}nal changes of diagrams. If $F$ commutes with $\simp$-homotopy 
colimits and with homotopy coproducts, then the induced morphism $F:\mc{C}^{I}[\mc{W}^{-1}]\rightarrow\mc{D}^I[\mc{W}^{-1}]$ is a right 
exact morphism of weak right derivators.  
\end{prop}

\begin{proof} By \cite[Proposition 2.6]{C} it suf{f}{i}ces to see that $F$ preserves homotopy colimits, which holds by corollary \ref{morphHCC}.
\end{proof}

We now particularize previous result to the setting of model categories. If $(\mc{M},\mc{W})$ and $(\mc{N},\mc{W})$ are model categories
and $F:\mc{M}\rightarrow \mc{N}$ preserves weak equivalences between cof{i}brant objects, 
recall that the left derived functor of $F$, $\mbb{L}F:\mc{M}[\mc{W}^{-1}]\rightarrow\mc{N}[\mc{W}^{-1}]$, 
exists and is given by the composition of $F$ with a functorial 
cof{i}brant replacement $Q:\mc{M}\rightarrow\mc{M}$. Then $\mbb{L}F$ is the localization of the relative functor
$FQ:(\mc{M},\mc{W})\rightarrow (\mc{N},\mc{W})$, that we also denote by $\mbb{L}F$.

\begin{cor}\label{FcommutesHocolim} Let $(\mc{M},\mc{W})$ and $(\mc{N},\mc{W})$ be model categories, and $F:\mc{M}\rightarrow \mc{N}$ be a functor that
preserves weak equivalences between cof{i}brant objects. If $\mbb{L}F$ commutes with $\simp$-homotopy colimits and with homotopy coproducts, then $\mbb{L}F$ 
commutes with all homotopy colimits. In other words, $\mbb{L}F^I:\mc{M}^{I}[\mc{W}^{-1}]\rightarrow\mc{N}^I[\mc{W}^{-1}]$ is a right 
exact morphism of derivators.
\end{cor}

\begin{proof} First of all, we observe that realizable homotopy coproducts on a model category are the composition of the usual coproduct with a cosimplicial 
replacement. Given a discrete category $\Lambda$, arguing as in the proof of Theorem \ref{MCHCbis} we deduce that 
$\mtt{hocolim}_{\Lambda}^{\mc{M}} = \mtt{hocolim}_{\Lambda}^{\mc{M}_c} Q$. Since on $(\mc{M}_c,\mc{W})$ the homotopy coproducts are the coproducts, then
$\mtt{hocolim}_{\Lambda}^{\mc{M}} = \amalg_{\Lambda} Q$ as stated. In addition, the adjunction relative natural transformations for 
$(\amalg_{\Lambda} Q , c_{\Lambda})$ are obtained combining those of $(\amalg_{\Lambda}, c_{\Lambda})$ on $\mc{M}_c$ with the natural weak equivalence
$\epsilon:Q\rightarrow 1_{\mc{M}}$. It follows that the canonical relative natural transformation 
$(\amalg_{\Lambda} Q) (F Q) \dashrightarrow (F Q) (\amalg_{\Lambda} Q)$ is a length-two zigzag 
$$(\amalg_{\Lambda} Q) (F Q) \stackrel{\sim}{\longleftarrow} T \longrightarrow   (F Q) (\amalg_{\Lambda} Q)$$
If $\mbb{L}F = F Q$ commutes with homotopy coproducts, the above zigzag is an isomorphism of $\mc{R}el\mc{C}at$, which means that 
$T \rightarrow   (F Q) (\amalg_{\Lambda} Q)$ is a pointwise weak equivalence as well. 
Given a small category $I$, we have the induced zigzag of pointwise weak equivalences
$$\theta: (\amalg^I Q) (FQ) \stackrel{\sim}{\longleftarrow} \amalg^I T \stackrel{\sim}{\longrightarrow} (FQ) (\amalg^I Q) $$
As we saw in Remark \ref{OtroHocolim}, realizable homotopy colimits on $\mc{M}$ and $\mc{N}$
are given by the formula $\mtt{hocolim}_{\simp} \amalg^I Q$.
To f{i}nish, the canonical relative natural transformation $\mtt{hocolim}_I^{\mc{N}}\,F Q\dashrightarrow FQ \, \mtt{hocolim}_I^{\mc{M}}$ factors as 
$$\xymatrix@M=4pt@H=4pt@C=45pt{ {} \mtt{hocolim}_{\simp}\amalg^I Q \,FQ  \ar@{-->}[r]^-{\mtt{hocolim}_{\simp}\theta} & \mtt{hocolim}_{\simp}  FQ \amalg^I Q
\ar@{-->}[r]  & FQ \,  \mtt{hocolim}_{\simp} \amalg^I Q}$$
and hence it is an isomorphism of $\mc{R}el\mc{C}at$.
\end{proof}


\end{document}